\documentclass{amsart}

\usepackage{amsmath}
\usepackage{amsfonts}
\usepackage{amssymb}
\usepackage{amsthm}

\usepackage[greek,french,english]{babel}
\usepackage[T1]{fontenc}
\usepackage[latin1]{inputenc}
\usepackage{graphics}
\usepackage{enumerate}

\newcommand{\R}{\mathbb R}
\newcommand{\N}{\mathbb N}
\newcommand{\Z}{\mathbb Z}
\newcommand{\T}{\mathbb T}
\newcommand{\e}{\varepsilon}
\newcommand{\be}{\begin{equation}}
\newcommand{\ee}{\end{equation}}
\newcommand{\dv}{\mathrm{div}}
\newcommand{\p}{\partial}

\newcommand{\sgn}{\mathrm{sgn}}

\newcommand{\bu}{\bar{u}}

\newcommand{\ds}{\displaystyle}
\newcommand{\HP}{H^1_{\text{per}}(\T^N)}

\newcommand{\mean}[1]{\left\langle #1\right\rangle}

\newcommand{\cA}{\mathcal A}
\newcommand{\cG}{\mathcal G}

\newtheorem{prop}{Proposition}[section]
\newtheorem{lem}{Lemma}[section]
\newtheorem{lemA}{Lemma}

\newtheorem{thm}{Theorem}[section]

\newtheorem{corol}{Corollary}[section]

\theoremstyle{definition}
\newtheorem{rem}{Remark}

\newtheorem*{ex}{Example}

\begin{document}
\title{Long time behaviour of viscous scalar conservation laws}

\author{Anne-Laure Dalibard}
\address{DMA/CNRS, Ecole Normale Sup\'erieure, 45 rue d'Ulm, 75005 Paris, FRANCE. e-mail: {\tt Anne-Laure.Dalibard@ens.fr}}

\bibliographystyle{amsplain}
\maketitle
\begin{abstract}
This paper is concerned with the stability of stationary solutions of the conservation law $\partial_t u + \mathrm{div}_y A(y,u) -\Delta_y u=0$, where the flux $A$ is periodic with respect to its first variable. Essentially two kinds of asymptotic behaviours are studied here: the case when the equation is set on $\R$, and the case when it is endowed with periodic boundary conditions. In the whole space case, we first prove the existence of viscous stationary shocks - also called standing shocks - which connect two different periodic stationary solutions to one another. We prove that standing shocks are stable in $L^1$, provided the initial disturbance satisfies some appropriate boundedness conditions. We also extend this result to arbitrary initial data, but with some restrictions on the flux $A$.
In the periodic case, we prove that periodic stationary solutions are always stable. The proof of this result relies on the derivation of uniform $L^\infty$ bounds on the solution of the conservation law, and on sub- and super-solution techniques.

\end{abstract}



\section{Introduction}

This paper is devoted to the analysis of the long-time behaviour of the solution $u\in\mathcal C([0,\infty), L^1_{\text{loc}}(Q))\cap L^\infty_\text{loc}([0,\infty), L^\infty(Q))$ of the equation
\be\begin{aligned}
&\p_t u + \dv_y A(y,u) -\Delta_y u=0,\quad t>0, y\in Q,\\
&u_{|t=0}=u_0\in  L^\infty (Q).
\end{aligned}\label{eq:evol_per}
\ee
Above, $Q$ denotes either $\R$ or $\T^N$, the $N$-dimensional torus ($\T^N=\R^N/ [0,1)^N$), and $A\in W^{1,\infty}_\text{loc}(\T^N\times \R)^N$ is an $N$-dimensional flux (with $N=1$ when $Q=\R$).

Heuristically, it can be expected that the parabolicity of equation \eqref{eq:evol_per}  may yield some compactness on the trajectory $\{ u(t)\}_{t\geq 0}.$ Hence, it is legitimate to conjecture that the family $u(t)$ will converge as $t\to\infty$ towards a stationary solution of \eqref{eq:evol_per}. Such a result was proved when $Q=\T^N$ by the author in \cite{initiallayer} for a certain class of initial conditions, namely when $u_0$ is bounded from above and below by two stationary solutions of \eqref{eq:evol_per}. Such an assumption is in fact classical in the framework of conservation laws which admit a comparison principle: we refer for instance to \cite{Bonforte_al}, where the authors study the long time behaviour of the fast diffusion equation, and assume that the initial data is bounded by two Barenblatt profiles. The same kind of assumption was made in the context of travelling waves by Stanley Osher and James Ralston in \cite{OR}; let us also mention the review paper by Denis Serre \cite{SerreHandbook}, which is devoted to the stability of shock profiles of scalar conservation laws, and in which the author assumes at first that the initial data is bounded from above and below by shifted shock profiles.
Nonetheless, in \cite{FS} (see also \cite{SerreL1,SerreHandbook}), Heinrich Freist\"uhler  and Denis Serre  remove this hypothesis, and prove that shock stability holds under a mere $L^1$ assumption on the initial data.

The goal of this paper is to extend the result of \cite{initiallayer} to arbitrary initial data, that is, to prove that solutions of \eqref{eq:evol_per} converge towards a stationary solution for any initial data $u_0\in L^\infty(\T^N)$. We also tackle similar issues on the stability of stationary shock profiles in dimension one, when the equation is set on the whole space case ($Q=\R$). Thus, a large part of the paper is devoted to the proof of the existence of shock profiles, and to the analysis of their properties. We will see that the question of shock stability reduces in fact to the stability of periodic stationary solutions  of \eqref{eq:evol_per} in $L^1(\R)$. Unfortunately, we have been unable to prove that periodic stationary solutions of \eqref{eq:evol_per} are stable in $L^1(\R)$ for arbitrary fluxes. Hence we have left this issue open, and we hope to come back to it in a future paper.

The proof of stability in the periodic setting relies strongly on the derivation of uniform $L^\infty$ bounds on the family $\{u(t)\}_{t\geq 0}$. In the whole space case, the first step of the analysis is to prove the property for initial data which are bounded from above and below by viscous shocks; in fact, this result is similar to the one proved in \cite{initiallayer}, and  uses arguments from  dynamical systems theory, following an idea by S. Osher and J. Ralston \cite{OR} (see also \cite{SerreHandbook,AmadoriSerre}). But the derivation of uniform $L^\infty$ bounds is not sufficient to obtain a general stability result in the whole space case. Thus the idea is to use existing results on the long time behaviour of transport equations, which were obtained by Adrien Blanchet, Jean Dolbeault and Michal Kowalczyk in \cite{BDK}, and to apply those in the present context.
\vskip2mm

Throughout the paper, we use the following notation: if $v\in L^1(\T^N)$,
$$
\mean{v}=\int_{\T^N} v.
$$
We denote by $L^1_0(Q)$ the set of intergrable functions with zero mass:
$$
L^1_0(Q):=\{u\in L^1(Q), \ \int_Q u=0\}.
$$

Following \cite{ladypara}, for $\alpha\in(0,1)$, we define, if $I$ is an interval in $(0,\infty)$ and $\Omega$ is a domain in $\R^N$,
$$
H^{\frac{\alpha}{2},\alpha}(I\times \Omega)=\{u\in\mathcal C(\bar I\times \bar \Omega),\,\ \| u\|_{H^{\alpha/2,\alpha}(I\times \Omega)}<\infty \},
$$
where
\begin{eqnarray*}
 \| u\|_{H^{\frac{\alpha}{2},\alpha}(I\times \Omega)}&:=&\max_{(t,x)\in \bar I\times \bar \Omega}|u(t,x)|\\&+& \sup_{\substack{(x,t)\in I\times \Omega,\\(x',t')\in I\times \Omega,\\ |t-t'|\leq \rho}} \frac{|u(t,x)-u(t',x)|}{|t-t'|^{\alpha/2}}+\sup_{\substack{(x,t)\in I\times \Omega,\\(x',t')\in I\times \Omega,\\ |x-x'|\leq \rho}} \frac{|u(t,x)-u(t,x')|}{|x-x'|^\alpha};
\end{eqnarray*}
above, $\rho$ is any positive number.
We also set
$$
\mathcal C^\alpha (\Omega):=\left\{ u\in\mathcal C(\bar \Omega),\ \sup_{(x,x')\in\Omega^2}\frac{|u(x)-u(x')|}{|x-x'|^\alpha}<+\infty \right\}.
$$
Eventually, for $f\in L^1_\text{loc}(\R)$, $h\in \R$, we set $\tau_h f= f(\cdot +h)$.

\section{Main results}
\label{sec:results}

Before stating our main results, we recall  general features of equation \eqref{eq:evol_per}, together with some facts related to the stationary solutions of this equation.

In the rest of this paper, we denote by $S_t$ the semi-group associated with equation \eqref{eq:evol_per}. Notice that $S_t$ is always well-defined on $L^\infty(Q)$, thanks to the papers by Kru{\v{z}}kov \cite{kruzkhov1,kruzkhov2}. Moreover, we recall that the following  properties hold true (these are called the \textbf{Co-properties} in \cite{SerreHandbook}):
\begin{itemize}
 \item \textbf{Comparison:} if $a,b\in L^\infty(Q)$ are such that $a\leq b$, then $S_t a \leq S_t b$ for all $t\geq 0$.
\item \textbf{Contraction:} if $a,b\in L^\infty(Q)$ are such that $a-b\in L^1(Q)$, then $S_t a - S_t b \in L^1(Q)$ for all $t\geq 0$ and
$$
\|S_t a -S_t b \|_{L^1}\leq \| a-b\|_{L^1}\quad\forall t\geq 0.
$$

\item \textbf{Conservation:} if $a,b\in L^\infty(Q)$ are such that $a-b\in L^1(Q)$, then $S_t a - S_t b \in L^1(Q)$ for all $t\geq 0$ and
$$
\int_{Q}(S_t a -S_t b)=\int_Q(a-b)\quad\forall t\geq 0.
$$
\end{itemize}
Thanks to the Contraction property, the semi-group $S_t$ can be extended on $L^\infty(Q)+ L^1(Q)$.
The so-called ``Constant property'' in \cite{SerreHandbook} is not true in the present setting, since the flux $A$ does not commute with translations. In other words, constants are not stationary solutions of equation \eqref{eq:evol_per} in general. The existence of periodic (in space) stationary solutions of \eqref{eq:evol_per} was proved by the author in \cite{homogpara}, and we recall the corresponding result below:

\begin{prop}

Let $A\in W_{\text{loc}}^{1,\infty}(\T^N\times\R)^N$. Assume that
there exist $C_0>0$, $m\in [0,\infty)$, $n\in[0,\frac{N+2}{N-2})$
when $N\geq 3$, such that for all $(y,p)\in \T^N \times \R$
\begin{gather}
|\p_p A_i(y,p)|\leq C_0 \left(1 + |p|^m \right)\quad \forall\ 1\leq
i\leq N,\label{polygrowthai} \\
\left|\dv_y A (y,p)\right|\leq
C_0 \left(1 + |p|^n \right).\label{polygrowthaN}
\end{gather}
Assume as well that one of the following conditions holds:
\begin{gather}
m=0\ \text{or}\  0\leq
n<1  \ \text{or}\ \left( n<\min\left(\frac{N+2}{N},2\right)\ \text{and}\
\exists p_0\in \R,\ \dv_y A(\cdot,p_0)=
0\right).\label{hyp:ex_cell}
\end{gather}

Then for all $p\in\R$, there exists a unique solution $v(\cdot, p)\in \HP$
of the equation \be -\Delta_y v(y,p) + \dv_y
A(y,v(y,p))=0, \quad \mean{v(\cdot,p)}=p.\label{cell}\ee

The family $(v(\cdot, p))_{p\in\R}$ satisfies the following properties:

\begin{enumerate}[\bf (i)]
 \item Regularity estimates:
For all $p\in\R$, $v(\cdot,p)$ belongs to
$W^{2,q}_{\text{per}}(\T^N)$ for all $1<q<+\infty$ and additionally\be\forall R>0\quad \exists C_R>0\quad \forall p\in[-R,R]\quad ||v(\cdot,p)||_{W^{2,q}(\T^N)}\leq C_R.\label{estW1q}\ee 
\item Growth property: if $p>p'$, then
$$
v(y,p)>v(y,p')\quad\forall y\in \T^N.
$$

\item Behaviour at infinity: assume that 
\be
\sup_{v\in\R}\left\| \p_v A(\cdot, v) \right\|_{L^\infty(\T^N)}<+\infty.\label{hyp:unif_lip}
\ee
Then 
$$
\lim_{p\to-\infty} \sup_{y\in \T^N} v(y,p) = - \infty,\quad
\lim_{p\to+\infty} \inf_{y\in \T^N} v(y,p) = + \infty.
$$

\end{enumerate}

\label{prop:cellpb}\end{prop}

\subsection{\textit{A priori} bounds for solutions of scalar conservation laws}

Our first result is concerned with the derivation of \textit{a priori} bounds in $L^\infty$ which are uniform in time. Notice that such a result is not trivial  in general: in the homogeneous case, that is, when the flux $A$ does not depend on the space variable $x$, this result follows from the comparison principle stated earlier. However, in the present case, this argument does not hold, since constants are not stationary solutions of \eqref{eq:evol_per}. Of course, if there exists a constant $C$ such that $u_0\leq v(\cdot, C)$, then the comparison principle entails that $S_t u_0 \leq v(\cdot, C).$ Hence, the derivation of \textit{a priori} bounds is easy when  the initial data is bounded from above and below by solutions of equation \eqref{cell}. Consequently, the goal of this paragraph is to present similar results when the initial data does not satisfy such an assumption.

\begin{prop}
Assume that the flux $A$ satisfies the assumptions of Proposition \ref{prop:cellpb}. Assume also that for all $K>0$, there exists a positive constant $C_K$, such that for all $v\in[-K,K]$, for all $w\in\R$,
\be
 \begin{aligned}
\left| \dv_y A(y,v+ w)- \dv_y A(y,v)  \right|\leq C_K (|w| + |w|^n),\\
\left| \p_v A(y,v+ w)- \p_v A(y,v)  \right|\leq C_K (|w| + |w|^n),
 \end{aligned}
\label{hyp:A1}\ee
where $n<(N+2)/N$.

Let $u_0\in L^\infty(Q)$, and assume that there exists a stationary solution $U_0\in W^{1,\infty}(Q)$ of \eqref{eq:evol_per} such that $u_0\in U_0 + L^1(Q).$

Then  
$$
\sup_{t\geq 0}\| S_t u_0\|_{L^\infty(Q)}<+\infty.
$$
\label{prop:bounds}

\end{prop}

Notice that in the above proposition, we do not assume that the stationary solution $U_0$ is periodic. Thus $U_0$ is not necessarily a solution of equation \eqref{cell}, and may be, for instance, a viscous shock profile (see Proposition \ref{prop:ex_shock} below). In the periodic case, any function $u_0\in L^\infty$ is such that $u_0 - v(\cdot, 0)\in L^1(\T^N)$, and thus the result holds for all functions in $L^\infty$.

\subsection{Stability of stationary periodic solutions in the periodic case}

The derivation of uniform \textit{a priori} bounds for the solutions of equation \eqref{eq:evol_per} allows us to extend the stability results proved in \cite{initiallayer} to general classes of initial data. Let us first recall the stability result of \cite{initiallayer}:

\begin{prop}
Assume that the flux $A$ satisfies the assumptions of Proposition \ref{prop:cellpb}. Let $u_0\in L^\infty(\T^N)$ such that there exists $\beta_1,\beta_2\in\R$ such that
\be
v(\cdot, \beta_1)\leq u_0\leq v(\cdot, \beta_2).\label{hyp:initiallayer}
\ee

Then as $t\to\infty$
$$
S_t u_0 \to v(\cdot, \mean{u_0})\quad\text{in } L^\infty(\T^N).
$$

\label{prop:initiallayer}
\end{prop}

It was also proved in \cite{initiallayer} that under additional regularity assumptions on the flux $A$, the speed of convergence is exponential, due to a spectral gap result.

We now remove assumption \eqref{hyp:initiallayer} thanks to Proposition \ref{prop:bounds}:

\begin{thm}
Assume that the flux $A$ satisfies the assumptions of Proposition \ref{prop:cellpb}, together with \eqref{hyp:A1}. Then for all $u_0\in L^\infty(\T^N)$, as $t\to\infty$,
$$
S_t u_0 \to v(\cdot, \mean{u_0})\quad\text{in } L^\infty(\T^N).
$$

\label{thm:lgtime}

\end{thm}
The proof of this result relies mainly on Proposition  \ref{prop:bounds} and on sub- and super-solution methods based on the Comparison principle. Once again, it can be proved that the speed of convergence is exponential, provided the flux $A$ is sufficiently smooth. For details regarding that point, we refer to \cite{initiallayer}.

\subsection{Existence of viscous shocks}

We now consider equation \eqref{eq:evol_per} set in $Q=\R$. Our goal here is to prove the stability of a special class  of stationary solutions, called ``standing shocks''. By analogy with the definition in \cite{SerreHandbook} of shocks in homogeneous conservation laws, a \textit{standing shock} is a stationary solution $U$ of equation \eqref{eq:evol_per} which is asymptotic to solutions of equation \eqref{cell} at infinity, namely
$$
\exists (p_l,p_r)\in\R^2,\quad \lim_{y\to-\infty}(U(y) - v(y,p_l))=0,\ \lim_{y\to+\infty}(U(y) - v(y,p_r))=0.
$$

Because of the spatial dependence of the flux $A$, it does not seem to be possible to restrict the study of general shocks to standing shocks. For that matter, we wish to emphasize that the definition of a viscous shock with non-zero speed should not be exactly the same as in \cite{SerreHandbook}; indeed, it can be easily checked that if 
$$
u(t,x)=U(x-st)
$$
is a solution of \eqref{eq:evol_per}, then $s=0$ necessarily. Thus, for $s\neq 0$, a shock profile is a solution of \eqref{eq:evol_per} of the form $$
u(t,x)=U(t,x-st),
$$
where for all $t$, $U(t)$ is asymptotic to solutions of equation \eqref{cell} at infinity. This is related (although not equivalent to) the definition of traveling pulsating fronts, see for instance the paper of Xue Xin \cite{XXin} The existence of non-stationary shock profiles and their stability is beyond the scope of this paper, and thus, we will focus on standing shocks from now on, sometimes omitting the word ``standing''.

Our first result is concerned with the existence of viscous shocks.

\begin{prop}[Existence of stationary shock profiles]
Assume that there exists $p^-, p^+\in\R$  such that $p^-<p^+$ and
\be
\bar A(p^+)=\bar A(p^-):=\alpha,
\label{hyp:plpr}
\ee
and define $v_\pm:=v (\cdot, p^\pm).$

Let $u_0\in \R$ such that
$$
v_-(0)<u_0< v_+(0),
$$
and let $u:I\to \R$ be the maximal solution of the differential equation
\begin{eqnarray}
 \label{eq:visc-shock}&&u' (x)= A(x,u(x)) - \alpha,\\
&&u_{|x=0}=u_0. \label{CI:visc-shock}
\end{eqnarray}
Then $u$ satisfies the following properties:
\begin{enumerate}[\bf (i)]
\item The function $u$ is a global solution of \eqref{eq:visc-shock}; in other words, $I=\R$.
\item For all $x\in \R$,
$$
v_-(x) \leq u(x) \leq v_+(x);
$$

\item There exist $q_l, q_r\in \R$ such that $q_l\neq q_r$, $\bar A(q_l)=\bar A(q_r)=\alpha$, and
$$
\lim_{x\to-\infty}(u(x) - v(x, q_l))=0,\quad\lim_{x\to+\infty}(u(x) - v(x, q_r))=0.
$$

\end{enumerate}
As a consequence, the solution $u$ of \eqref{eq:visc-shock}-\eqref{CI:visc-shock} is a stationary viscous shock profile of equation \eqref{eq:evol_per}.

\label{prop:ex_shock}

\end{prop}

\begin{rem}
\begin{enumerate}[\bf (i)]
 \item Assumption \eqref{hyp:plpr} is the analogue of the Rankine-Hugoniot condition for homogeneous conservation laws. It is in fact a necessary condition, as demonstrated in Lemma \ref{lem:CN} below.

\item In general, the asymptotic states $v(q_l), v(q_r)$ are different from $v(p^+), v(p^-)$. Proposition \ref{prop:ex_shock} only ensures that
$$
\bar A(q_l)=\bar A(q_r)=\bar A(p^\pm).
$$
However, under an additional condition of Oleinik type, the asymptotic states can be identified:

\end{enumerate}

\end{rem}

\begin{corol}
Assume that the hypotheses of Proposition \ref{prop:ex_shock} are satisfied, and that the flux $\bar A$ is such that
\be\label{hyp:Oleinik}
\forall p \in (p^-,p^+),\ \bar A(p)\neq \alpha.
\ee
Then $$\{ q_l, q_r\}=\{ p^+,p^-\}. $$
\label{cor:oleinik}
\end{corol}

\subsection{Stability of stationary shocks in the whole space case}

We are now ready to state results on shock stability for equation \eqref{eq:evol_per}. Our first result is the analogue of Proposition \ref{prop:initiallayer}: indeed, Theorem \ref{thm:stab_shock1} below states that $S_t u_0$ converges towards a viscous shock, provided $u_0$ is bounded from above and below by the asymptotic states of the shock. In view of Theorem \ref{thm:lgtime}, it is natural to expect that this result remains true for arbitrary initial data. Unfortunately, we have not been able to prove this result in complete generality: we prove that stationary shocks are stable in $L^1$ provided stability holds (in $L^1(\R)$) for solutions of equation \eqref{cell}. We also give explicit examples of fluxes for which the stability of shocks and periodic stationary solutions can be established.

\begin{thm}Assume that the flux $A$ satisfies the assumptions of Proposition \ref{prop:cellpb}. Let
$p_l,p_r\in \R$ such that $p_l\neq p_r$ and $
                                 \bar A(p_r)=\bar A(p_l)=:\alpha$, and assume that  $\bar A, p_l, p_r$ satisfy Oleinik's condition \eqref{hyp:Oleinik}.

Let $U$ be a shock profile connecting $v(p_l)$ to $v(p_r)$. Let $u_0\in U + L^1(\R)$ such that 
for almost every $x\in\R$,
\be
v(x,\min(p_l,p_r))\leq u_0(x)\leq v(x,\max(p_l,p_r)).\label{hyp:u_0}
\ee
Then there exists a shock profile $V$ connecting $v(p_l)$ to $v(p_r)$
and such that $u\in V + L^1_0(\R).$ Moreover,
$$
\lim_{t\to\infty}\left\| S_t u_0- V \right\|_{L^1(\R)}=0.
$$
\label{thm:stab_shock1}
\end{thm}

As outlined before, hypothesis \eqref{hyp:u_0} should be compared with assumption \eqref{hyp:initiallayer}. Thus, the next step would be to prove that stability holds even when \eqref{hyp:u_0} is false. In fact, we are only able to prove the following:

\begin{prop}
 Assume that the flux $A$ satisfies the assumptions of Proposition \ref{prop:bounds}. Let $p_l,p_r\in \R$ such that $p_l\neq p_r$, $\bar A(p_r)=\bar A(p_l)$, and such that  \eqref{hyp:Oleinik} is satisfied.
Assume that the following assertion is true:

\begin{center}
 {\bf(H)}                     For  $p\in\{p_l,p_r\}$, for all $u_0 \in v(\cdot, p) + L^1_0(\R)$, $
\ds\lim_{t\to\infty}\| S_t u_0 - v(\cdot,p)\|_{L^1(\R)}=0.
$
\end{center}

Let $U$ be a shock profile connecting $v(p_l)$ to $v(p_r)$, and let $u_0\in U + L^1(\R)$. 
Then there exists a shock profile $V$ connecting $v(p_l)$ to $v(p_r)$
and such that $u\in V + L^1_0(\R).$ Moreover,
$$
\lim_{t\to\infty}\left\| S_t u_0- V \right\|_{L^1(\R)}=0.
$$

\label{prop:stab_shock2}

\end{prop}

\begin{rem}

$\bullet$ As we will see in Section \ref{sec:shock_stab:2}, hypothesis  \textbf{(H)} can be relaxed into 
\vskip1mm

 \quad\textbf{(H')} For  $p\in\{p_l,p_r\}$, there exists $\delta>0$ such that for all $u_0\in v(\cdot, p) + L^1_0(\R) $ ,
$$
\|u_0 - v(p) \|_1\leq \delta\Rightarrow\lim_{t\to\infty}\| S_t u_0 - v(\cdot,p)\|_{L^1(\R)}=0.
$$

\noindent$\bullet$ In Section \ref{sec:shock_stab:2}, we will prove the following result: for all $p\in\R$, there exists $\delta>0$ such that if $u_0\in v(\cdot, p) + L^1_0(\R) $ satisfies $\|u_0 - v(p) \|_1\leq \delta$, then 
$$
\lim_{t\to\infty}\| S_t u_0 - v(\cdot,p)\|_{L^\infty(\R)}=0.
$$
Hence, in this case, for all $\e>0$ there exists $t_\e\geq 0$ such that $S_{t_\e} u_0\in v(\cdot, p) + L^1_0$ and
$$
v(\cdot, p-\e)\leq u_0\leq v(\cdot, p+\e).
$$
Consequently, hypothesis \textbf{(H)} can also be relaxed into
\vskip1mm

\quad
 \textbf{(H'')} 
For all $p\in\R$, there exists $\delta>0$ such that for all $u_0\in v(\cdot, p) + L^1_0(\R) $ 
$$
v(\cdot, p-\delta)\leq u_0\leq v(\cdot, p+\delta)\Rightarrow\lim_{t\to\infty}\| S_t u_0 - v(\cdot,p)\|_{L^1(\R)}=0.
$$

To sum up, denoting by \textbf{(C)} the conclusion of Proposition \ref{prop:stab_shock2} (that is, shock stability), we have roughly
$$
\textbf{(H)} \Rightarrow \textbf{(H'')} \Rightarrow \textbf{(H')}\Rightarrow  \textbf{(C)} .
$$

\end{rem}
\vskip2mm

We now give an example when it is known that \textbf{(H)} is true. This example relies on the analysis performed in the linear case by A. Blanchet, J. Dolbeault and M. Kowalczyk (see \cite{BDK}).
\begin{prop}
Assume that the flux $A$ satisfies the hypotheses of Proposition \ref{prop:bounds}, and let $p\in\R$. Assume that $A$ is linear in a neighbourhood of $v(\cdot, p)$, i.e.
$$
\exists b\in \mathcal C^1(\T^N),\ \exists \eta>0, \ \forall \xi\in(-\eta,\eta),\ A(y,v(y,p)+ \xi)= A(y,v(y,p)) + b(y)\xi.
$$ 
Then, provided a technical assumption on the moments of $S_t u_0$ is satisfied (see \eqref{moments}), there exists $\delta>0$ such that for all $u_0\in v(\cdot, p )+ L^1_0,$
$$
\| u_0-v(\cdot,p)\|_1\leq \delta\Rightarrow \lim_{t\to\infty}\|S_t u_0 - v(p) \|_1=0.
$$
\label{prop:linear}
\end{prop}
The assumption \eqref{moments} is a little technical to state at this stage, and is inherited from the analysis in \cite{BDK}. However, as explained in \cite{BDK}, this hypothesis is expected to be satisfied for a large class of initial data, so that in fact \eqref{moments} is not a restriction.

This allows us to give an explicit case of flux for which shock stability holds.
\begin{corol}
Assume that the flux $A$ is given by
$$
A(y,v)=V(y) + f(v),
$$
 where $V\in\mathcal C^2(\T^N)$ and $f\in\mathcal C^2(\R)$ is a convex function which is linear at infinity, i.e.
$$
\exists (a_-,a_+)\in(0,\infty)^2,\ \exists A>0\quad \left\{\begin{array}{ll}
f(v)=a_+ v& \text{ if }v>A,\\
f(v)=-a_- v& \text{ if }v<-A.
\end{array}
 \right.
$$
Then the following properties hold:
\begin{enumerate}[\bf (i)]
 \item For all $\alpha>0$ large enough, there exist $(p_l,p_r)\in\R^2$ such that $p_l\neq p_r$ and $\bar A(p_l)=\bar A(p_r)=\alpha$, and there exists a shock profile $U$ connecting $v(p_l)$ to $v(p_r)$.

\item Let $u_0\in U + L^1(\R)$. Then there exists a shock profile $V$ such that $u\in V + L^1_0$. Moreover, if \eqref{moments} holds for any $v_0\in v(p^\pm) + L^1_0(\R)$, then $\lim_{t\to\infty}\| S_t u_0 - V\|_1=0.$

\end{enumerate}

\label{cor:convex}

\end{corol}

The plan of the paper is the following: given the similarity between the statements for periodic solutions when $Q=\T^N$, and stationary shocks when $Q=\R$, we first prove the existence of standing shocks (i.e. Proposition \ref{prop:ex_shock}) and the shock stability result under boundedness conditions on the initial data (i.e. Theorem \ref{thm:stab_shock1}) in sections \ref{sec:shocks_ex} and \ref{sec:shock_stab:1} respectively. At this stage, we are able to treat both models simultaneously, and thus we prove Proposition \ref{prop:bounds} in Section \ref{sec:bounds}. Section \ref{sec:per-stab} is devoted to the proof of Theorem \ref{thm:lgtime}, and at last, we give in Section \ref{sec:shock_stab:2} further results on shock stability, including the proofs of Propositions \ref{prop:stab_shock2} and \ref{prop:linear}.

\section{Existence of one dimensional stationary viscous shocks}

\label{sec:shocks_ex}

This section is devoted to the proof of Proposition \ref{prop:ex_shock}, together with a number of results related to viscous shocks which will be useful in the proof of Theorem \ref{thm:stab_shock1}. These auxiliary results (monotonicity of shock profiles, integrability of the difference between two shock profiles, etc.) can be found in paragraph \ref{ssec:aux_shock}. We also give in paragraph \ref{ssec:convex} a few explicit examples in the case when the flux $A$ is convex. 

We begin with some comments on assumption \eqref{hyp:plpr}.

\subsection{Analysis of necessary conditions}

\begin{lem}
Let $q_l,q_r\in \R$, and let $u\in W^{1, \infty}(\R)$ be such that
$$
\begin{array}{l}
u(x) - v(x, q_l)\to 0 \quad \text{as }x\to - \infty,\\
u(x) - v(x, q_r)\to 0 \quad \text{as }x\to + \infty,\\
\ds - u''+ \frac{d}{dx}(A(x,u(x)))=0.
\end{array}
$$
Then $\bar A(q_r)=\bar A(q_l)=:\alpha,$ and $u$ satisfies
$$
- u' + A(x,u(x))=\alpha.
$$
\label{lem:CN}
\end{lem}

\begin{proof}
 We deduce from the differential equation that there exists a constant $C$ such that
$$
-u' + A(x,u)=C,
$$
and the goal is to prove that $\bar A (q_r) = C =\bar A (q_l)$. We recall first that for all $p\in \R$, $v(\cdot,p)$ is a solution of
$$
- v' (x, p) + A(x,v(x,p ))=\bar A (p).
$$
Indeed, integrating \eqref{cell} on $\R$, we infer that for all $p\in \R$ there exists a constant $C_p$ such that
$$
\forall x\in \R,\quad- v' (x, p) + A(x,v(x,p ))=C_p.
$$
Taking the average of the above equality over a period, we deduce that $C_p=\bar A(p)$.

As a consequence, we have
\be\label{eq:u-vr}
- \frac{d}{dx}\left( u (x) - v(x, q_r)\right) +\left[ A(x,u(x)) - A(x,v(x,q_r)) \right] = C- \bar A(q_r)
\ee
Now, let $\delta>0$ arbitrary. There exists $x_r>0$ such that
$$   x\geq x_r \Rightarrow \left( \left| u(x) - v(x, q_r) \right|\leq \delta,\quad \left|A(x, u(x)) -A(x,v(x, q_r)) \right|\leq \delta\right).
$$
Integrating \eqref{eq:u-vr} on the interval $[x_r, x_r + 1]$, we deduce that
$$
|C- \bar A(q_r)|\leq 3 \delta.
$$
Since the above inequality is true for all $\delta>0$, we infer that $C=\bar A(q_r) $. The other equality is treated similarly.

\end{proof}

\begin{rem}
Notice that couples $(p_l,p_r)$ such that $p_l\neq p_r$ and $\bar A(p_l)=\bar A(p_r)$ do not always exist. Indeed, consider the case of a linear flux $A(x,v)=a(x) v$, with $a\in\mathcal C^1(\T)$. Then, for all $p\in\R$, we have $v(x,p)=pm(x)$, where $m$ is the unique probability measure on $\T$ satisfying
$$
-m''(x) + \frac{d}{dx}(a(x) m(x))=0,\quad x\in \T.
$$ 
The positivity of $m$ is a consequence of the Krein-Rutman Theorem; we refer to \cite{homogpara} for more details.

Therefore, for all $p\in\R$,
$$
\bar A(p)=\mean{a v(\cdot, p)}=p \mean{am}.
$$
Hence, as long as $\mean{am}\neq 0$, $\bar A(p)\neq \bar A(q)$ for all $p,q\in\R$ such that $p\neq q$. In particular, if $a$ is a non-zero constant, assumption \eqref{hyp:plpr} is never satisfied.
\label{rem:linear}
\end{rem}

\subsection{Proof of Proposition \ref{prop:ex_shock}}

We begin with the \textit{a priori} bound \textbf{(ii)}, from which we deduce that $u$ is a global solution.

The inequality \textbf{(ii)} follows directly from classical results in differential equations; indeed, assume that there exists $x_0\in I$ such that
$$
u(x_0)\geq v_+(x_0);
$$
since $u(0)<v_+(0)$, there exists $x_1\in [0, x_0]$ such that $u(x_1)=v_+(x_1)$. But $u$ and $v_+$ are solutions of the same differential equation, whence the Cauchy-Lipschitz Theorem implies that $u=v_+$, which is false. Thus
$$
u(x) < v_+(x)\quad  \forall x\in I.
$$
The lower bound is proved in the same way.

As a consequence, we deduce that $u$ remains bounded on its (maximal) interval of existence $I$. Using once again the Cauchy-Lipschitz Theorem, we infer that $I=\R$, and thus $u$ is a global solution.

We now tackle the core of Proposition \ref{prop:ex_shock}. First, since the flux $A$ is $\T$-periodic, the function $u(\cdot +1)$ is also a solution of equation \eqref{eq:visc-shock}. Hence the function $x\mapsto u(x+1) - u(x)$ keeps a constant sign on $\R$, which entails in particular that for all $x\in \R$, the sequences $(u(x\pm k))_{k\in \N}$ are monotonous. Consider for instance the sequence of functions
$$
u_k: x\in[0,1]\mapsto u(x+k).
$$
According to the above remarks, the sequence $(u_k)$ is monotonous and bounded in $L^\infty$; hence for all $x\in [0,1]$, $u_k(x)$ has a finite limit, which we denote by $u_\infty(x)$. Moreover, thanks to the uniform bound \textbf{(ii)} and the differential equation \eqref{eq:visc-shock}, $u$ belongs to $W^{1,\infty}(\R)$, and thus the sequence $u_k$ is uniformly bounded (with respect to $k$) in $W^{1,\infty}([0,1])$. Thus $u_\infty\in W^{1,\infty}([0,1])$, and $u_\infty$ is a continuous function. According to Dini's Theorem, we eventually deduce that $u_k$ converges  towards $u_\infty$ in $L^\infty([0,1])$. Notice that $u_\infty$ is periodic by definition, and passing to the limit in equation \eqref{eq:visc-shock}, we deduce that $u_\infty$ is a solution of \eqref{eq:visc-shock}.
Hence $u_\infty$ belongs to $W^{1,\infty}(\T)$ and satisfies
$$
- u_{\infty}'' + \frac{d}{dx} (A(x, u_{\infty}(x))) =0,
$$
which means exactly that $u_{\infty}$ is a periodic solution of equation \eqref{cell}; according to Proposition \ref{prop:cellpb}, there exists $p\in \R$ such that $u_\infty=v(\cdot,p).$ Eventually, since $u_\infty$ is a solution of \eqref{eq:visc-shock}, we infer that $\alpha=\bar A(p).$ To sum up, we have proved that there exists $p\in [p^-,p^+]$, such that $\bar A(p)=\bar A (p^\pm)$, and such that
$$
\lim_{k\to\infty} \sup_{x\in [0,1]}\left|u (x+k) - v(x,p) \right| =0.
$$
The above convergence is strictly equivalent to $u(x) - v(x,p)\to 0$ as $x\to\infty,$ and thus the third point of the Proposition is proved. The limit as $x\to-\infty$ is treated similarly.

\subsection{Further results on viscous shocks}

\label{ssec:aux_shock}

We have gathered in this paragraph some  results which will be important in the proof of Theorem \ref{thm:stab_shock1}. The first lemma gives a criterion allowing us to distinguish between the asymptotic states at $\pm\infty$.

\begin{lem}
Let $p_l,p_r\in\R$ such that $\bar A(p_l)=\bar A(p_r)$, and let $U$ be a shock profile such that
$$
\lim_{x\to-\infty}\left[ U(x)-v(x,p_l) \right]= \lim_{x\to+\infty}\left[ U(x)-v(x,p_r) \right]=0.
$$ 
Then 
$$
   \mean{ \p_v A(\cdot, v(\cdot,p_l))}\geq 0,\quad
   \mean{ \p_v A(\cdot, v(\cdot,p_r))}\leq 0.
$$
Moreover, if one of the above inequalities is strict, then  $U$ converges exponentially fast toward the corresponding solution of equation \eqref{cell}; for instance, if
$$
\bar a_r:=\int_{\T} \p_v A(y, v(y,p_r))\:dy <0,
$$
then for all $a\in(0,-\bar a_r)$, there exists a constant $C_a$ such that for all $y\in[0,\infty)$,
$$
\left| U(y) - v(y,p_r) \right|\leq C_a \exp(-ay).
$$
\label{prop:sens}

\end{lem}

\begin{proof}
 
Since $U$ is a shock profile and $v(p_l)$, $v(p_r)$ are solutions of equation \eqref{cell}, we have
$$\begin{aligned}
 U'(x)= A(x,U(x)) - \alpha,\\
\p_x v(x,p_l) = A(x,v(x,p_l)) - \alpha,\\
\p_x v(x,p_r) = A(x,v(x,p_r)) - \alpha,
\end{aligned}
$$
where $\alpha$ denotes the common value of $\bar A(p_l)$ and $\bar A(p_r)$.

Consequently, the function $U-v(p_r)$, for instance, satisfies the linear equation
\be
\p_x(U(x)- v(x,p_r))= b(x) (U(x)- v(x,p_r)),\label{eq:U-v_r}
\ee
where 
$$
b(x)=\int_0^1 \p_v A(x,\tau U(x) + (1-\tau) v(x,p_r))\:d\tau.
$$
Notice that since $U$ converges towards $v(p_r)$ as $x\to+\infty$, we obtain
\be
\lim_{x\to+\infty} \left[ b(x) -\p_v A(x,v(x,p_r))  \right]=0.\label{b_asympt1}
\ee

On the other hand, equation \eqref{eq:U-v_r} implies that
$$
U(x) - v(x,p_r)= \left[U(0) - v(0,p_r)\right]\exp\left( \int_0^x b(y)\:dy \right).
$$
Once again, since $U-v(p_r)$ converges towards zero, we infer that
\be
\lim_{x\to+\infty} \int_0^x b(y)\:dy =-\infty.\label{b_asympt2}
\ee
The first statement of the proposition follows easily from \eqref{b_asympt1}, \eqref{b_asympt2}; indeed, assume that $\bar a_r>0$. Then there exists a positive number $K$ such that
$$
x\geq K \Rightarrow b(x) -  \p_v A(x,v(x,p_r))\geq -\frac{\bar a_r}{2},
$$
and consequently, using the fact that $x\mapsto\p_v A(x,v(x,p_r))$ is a periodic function, we obtain for $x\geq K$
\begin{eqnarray*}
\int_K^x b(y) \:dy &\geq & \int_K^x \p_v A(y,v(y,p_r)) \:dy - (x-K)\frac{\bar a_r}{2}\\
&\geq & \lfloor x-K\rfloor \bar a_r- x\frac{\bar a_r}{2} -C\\
&\geq & x\frac{\bar a_r}{2} -C.
\end{eqnarray*}
The above inequality is obviously in contradiction with \eqref{b_asympt2}. Hence $\bar a_r\leq 0$, which proves the first statement in the proposition.

Now, assume that $\bar a_r< 0$, and choose $a\in (0, -\bar a_r)$ arbitrary. As before, we pick $K>0$ such that
$$
x\geq K \Rightarrow b(x) -  \p_v A(x,v(x,p_r))\leq -\bar a_r-a.
$$
We then obtain an inequality of the type
\begin{eqnarray*}
\int_K^x b(y) \:dy &\leq &  (-\bar a_r-a)(x-K) +  \lfloor x-K\rfloor \bar a_r + C\\
&\leq & -a x + C.
\end{eqnarray*}
Inserting this inequality back into the formula for $U-v(p_r)$ yields the exponential convergence result.

\end{proof}

The next result is concerned with the integrability of the difference between two shock profiles. 

\begin{lem}
Let $p_l,p_r\in \R$ such that $p_l\neq p_r$ and $\bar A(p_l)=\bar A(p_r)$, and let $U,V$ be two shock profiles with asymptotic states $v(p_l)$, $v(p_r)$.

Then $U-V\in L^1(\R).$

\label{lem:U-V_L1}
\end{lem}

\begin{proof}
Set
$$
U_0:=U(0),\ V_0:=V(0),
$$ 
and assume for instance that $U_0\leq V_0$. If $U_0=V_0$, then $U=V$ according to the Cauchy-Lipschitz Theorem (see the proof of Proposition \ref{prop:ex_shock}), and the result is obvious. Thus we assume from now on that $U_0<V_0$. As a consequence, we have
$$
\forall y\in \R,\quad v(y,\min(p_l, p_r)) < U(y) <V(y)<v(y,\max(p_l, p_r)).
$$
We recall that the sequence $(U(k))_{k\in\Z}$ is monotonous, and converges towards $v(0,p_l)$ (resp. $v(0,p_r)$) as $k\to -\infty$ (resp. $k\to +\infty$). Hence, there exists an integer $k_0\in \Z$ such that
\be
U_0< V_0 < U(k_0),
\label{in:UV}\ee
from which we infer that $U\leq V\leq \tau_{k_0} U$.

As a consequence, it is sufficient to prove that $\tau_{k} U-U$ is integrable, for all $k\in\Z$.

First, remember that $\tau_k U -U$ has a constant sign, since $\tau_k U$ and $U$ are both shock profiles. Thus we only have to prove that the family
$$
\int_{-A}^A (\tau_k U -U)
$$
remains bounded as $A\to \infty$. A simple calculation leads to
\begin{eqnarray*}
 \int_{-A}^A (\tau_k U -U)&=&\int_{-A}^A U(y+k) \:dy - \int_{-A}^A U(y) \:dy\\
&=&\int_{k-A}^{k+A} U(y) \:dy - \int_{-A}^A U(y) \:dy\\
&=&\int_A^{k+A}U(y) \:dy  - \int_{-A}^{k-A}U(y) \:dy .
\end{eqnarray*}
Thus, recalling that $U$ is a bounded function, we obtain
$$
\sup_{A>0}\left|\int_{-A}^A (\tau_k U -U) \right|\leq 2 k \|U \|_{L^\infty(\R)}.
$$
We deduce that $\tau_k U - U \in L^1(\R)$ for all $k\in\Z$, and eventually that $U-V\in L^1(\R)$ according to \eqref{in:UV}.

\end{proof}

The  next result is in fact the first part of the statement of Theorem \ref{thm:stab_shock1}:

\begin{lem} Let $p_l,p_r\in \R$ such that the assumptions of Theorem \ref{thm:stab_shock1} are satisfied, and let 
 $U$ be a viscous shock  connecting $v(p_l)$ to $v(p_r)$.

Let $u\in U + L^1$. Then there exists a unique shock profile $V$,  with asymptotic states $v(p_l)$ and $v(p_r)$, and such that
$$
u\in V + L^1_0(\R).
$$
\label{lem:L10}
\end{lem}

\begin{proof}According to Lemma \ref{lem:U-V_L1}, we already know that for every shock profile $V$, we have $u-V\in L^1.$ Hence, the question is to find a shock profile $V$ such that
\be
\int_{\R} (u-V)=0.\label{u-V}
\ee
Notice that such a shock profile is necessarily unique: indeed, the Cauchy-Lipschitz uniqueness principle entails that the difference of two shock profiles is a function which keeps a constant sign. Hence, if $V_1,V_2$ are shock profiles satisfying
$
\int_{\R}(V_1-V_2)=0,
$
then $V_1=V_2$.

We now prove that there exists a shock profile $V$ such that $u-V\in L^1_0(\R)$. As before, we set $p^-=\min(p_l,p_r)$, $p^+=\max(p_l,p_r)$.
For all $\xi\in(v(0,p^-), v(0,p^+))$, we denote by $V_\xi$ the solution of 
$$
\begin{aligned}
V'(x)=A(x,V(x)) - \bar A(p_l) ,\\
V_{|x=0}=\xi.
\end{aligned}
$$
Then, according to Proposition \ref{prop:ex_shock}  and Lemma \ref{lem:U-V_L1}, for all $\xi$, $V_\xi$ is a shock profile connecting $v(p_l)$ to $v(p_r)$, and additionally $u-V_\xi\in L^1(\R).$
Moreover, if $\xi>\xi'$, then $V_\xi(x)>V_{\xi'}(x)$ for all $x$; hence the function
$$
F:\xi\in (v(0,p^-), v(0,p^+)) \mapsto \int_{\R}(u(x)- V_\xi(x))\:dx
$$
is well-defined and decreasing with respect to $\xi$; using classical results on differential equations, it can easily be proved that $F$ is continuous. We wish to find $\xi_0$ such that $F(\xi_0)=0$; thus it suffices to show that
$$
\lim_{\xi \to  v(0,p^-)_+} F(\xi)>0\text{ and } \lim_{\xi \to  v(0,p^+)_-} F(\xi)<0.
$$
The above result is a direct consequence of Lebesgue's monotone convergence Theorem and of the fact that 
\be
\forall x\in\R,\quad \lim_{\xi\to v(0,p^-)_+}V_\xi(x)= v(x,p^-).\label{conv_Uxi}
\ee
The same kind of result holds with $v(p^+)$. Indeed, let $R>0$ be arbitrary, and let $\e>0$. Without loss of generality, assume that $p_r=p^-$. Then there exists $K\in \N$ such that  
$$
x\geq K\Rightarrow v(x,p_r)\leq U(x) \leq v(x,p_r )+ \e.
$$
In particular, $\tau_{K + \lfloor R\rfloor + 1} U$ is a shock profile which satisfies
$$
\tau_{K + \lfloor R\rfloor + 1} U(x) \leq v(x,p_r )+ \e\quad\forall x\in[-R,R].
$$
Let $\bar \xi:=\tau_{K + \lfloor R\rfloor + 1} U (0)=  U (K + \lfloor R\rfloor + 1)$. The Cauchy-Lipschitz Theorem entails that $V_{\bar \xi}= \tau_{K + \lfloor R\rfloor + 1} U$. As a consequence, for all $\xi<\bar \xi,$ for all $x\in[-R,R]$, we have
$$
v(x,p_r)\leq V_\xi(x)\leq V_{\bar \xi}(x)\leq v(x,p_r) + \e.
$$
The convergence result \eqref{conv_Uxi} follows, and thus there exists a shock profile $V$ such that $u_0\in V + L^1_0(\R).$

\end{proof}

The next lemma allows us to replace inequality \eqref{hyp:u_0} by an inequality in which the upper and lower bounds are shock profiles, which will be useful in the proof of Theorem \ref{thm:stab_shock1} in Section \ref{sec:shock_stab:1}.

\begin{lem}
Let $p_l,p_r$ such that the hypotheses of Theorem \ref{thm:stab_shock1} are satisfied.
 Let $U$ be a shock profile  connecting $v(p_l)$ to $v(p_r)$.
Let $u\in L^\infty(\R)$ such that $u\in U + L^1_0(\R)$ and assume that for almost every $y\in \R$,
$$
v(y,\min(p_r, p_l))\leq u(y)\leq v(y,\max(p_l,p_r)).
$$

Let $\e>0$ be arbitrary. Then there exists a function $u^\e\in U + L^1_0(R)$, together with shock profiles $U^\e_\pm$ connecting $v(p_l)$ to $v(p_r)$, such that
$$
\|u-u^\e\|_{L^1}\leq \e,\quad
U^\e_-\leq u^\e\leq U^\e_+.
$$

\label{lem:up/lowshocks}
\end{lem}

\begin{proof}
First, since $u-U\in L^1(\R)$, there exists a positive number $A^\e$ such that
$$
\int_{|x|\geq A^\e}|u-U|\leq \e.
$$
Hence, for $|x|\geq A^\e$, we take $u^\e(x)= U(x).$

The definition of $u^\e$ on the interval $[-A^\e, A^\e]$ is slighlty more technical, because of the various constraints bearing on $u^\e$.
Once again, we assume that $p_l>p_r$ in order to lighten the notation. We first consider a function $v^\e\in\mathcal C([-A^\e,A^\e])$ which satisfies
$$
\int_{|x|\leq A^\e} |u(x) - v^\e(x)|\:dx\leq \e
$$
and such that
$$
v(x,p_r)<v^\e(x)<v(x,p_l)\quad \forall x\in [-A^\e, A^\e].
$$
We denote by $\alpha^\e$ a positive number such that
$$
v(x,p_r) + \alpha^\e\leq v^\e(x) \leq v(x,p_l)-\alpha^\e\quad\forall x\in[-A^\e, A^\e].
$$
Notice that $\alpha^\e$ can be chosen as small as desired. For further purposes, we choose $\alpha^\e$ so that $$\alpha^\e A^\e\leq 2 \int_{|x|\leq A^\e}(U-v(p_r)).$$

The constraint $u^\e\in U + L^1_0(\R)$ entails that the function $u^\e$ should satisfy
$$
\int_{|x|\leq A^\e} (u^\e-U)=0.
$$
However, the function $v^\e$ does not satisfy the above constraint
in general: we merely have
\begin{eqnarray*}
\left|\int_{|x|\leq A^\e} (v^\e-U)\right|&\leq& \left|\int_{|x|\leq A^\e} (v^\e - u) \right| +  \left|\int_{|x|\leq A^\e} (u- U) \right|\\
&\leq & \int_{|x|\leq A^\e}\left| v^\e - u \right| +\int_{|x|\geq A^\e} \left|u- U \right|\\
&\leq &2\e.
\end{eqnarray*}
Assume for instance that $\int_{|x|\leq A^\e} (v^\e-U)>0.$ We then define a non-negative function $\rho^\e\in L^\infty([-A^\e,A^\e])$ such that
\be
\begin{aligned}
v^\e(x) - \rho^\e(x) \geq  v(x, p_r) + \frac{\alpha^\e}{2}\quad \text{a.e. on }[-A^\e, A^\e]\\
\text{and}\quad \int_{|x|\leq A^\e} (v^\e-\rho^\e-U)=0.
\end{aligned}
\label{cond_rhoe}\ee
Such a function $\rho^\e$ exists provided
$$
\int_{|x|\leq A^\e}(v^\e-U)\leq \int_{|x|\leq A^\e} \left(v^\e - v(p_r) - \frac{\alpha^\e}{2}\right),
$$
and the above inequality is equivalent to
$$ \int_{|x|\leq A^\e} (U-v(p_r))\geq \frac{\alpha^\e A^\e}{2}.$$
The previous condition is satisfied by definition of $\alpha^\e$. Thus there exists a function $\rho^\e$ which satisfies conditions \eqref{cond_rhoe}.

We then set
$$
u^\e(x)= v^\e(x) - \rho^\e(x)\quad \text{for }x\in[-A^\e,A^\e].
$$
The construction is similar when $\int_{|x|\leq A^\e} (v^\e-U)<0.$

At this stage, we have defined a function $u^\e\in U + L^1_0$ which satisfies
$$
\begin{aligned}
 v(x,p_r) + \frac{\alpha^\e}{2}< u^\e(x) \leq v(x,p_l) - \frac{\alpha^\e}{2} \quad \forall x\in [-A^\e,A^\e],\\
u^\e(x)=U(x)\quad \forall x\in \R\setminus [-A^\e,A^\e],\\
\text{and}\quad\int_{\R} |u-u^\e|\leq 4\e.
\end{aligned}
$$
Now, by definition of the shock profile $U$, there exists a positive constant $R^\e$ such that
$$
x\geq R^\e\Rightarrow \left|U(x)- v(x,p_r)\right|\leq \frac{\alpha^\e}{2}.
$$
Let $k^+$ be a positive integer such that $k^+> R^\e + A^\e.$ Then for all $x\in [-A^\e,A^\e]$, we have
$$
v(x,p_r)\leq \tau_{k^+} U(x)\leq v(x,p_r) + \frac{\alpha^\e}{2} \leq u^\e(x).
$$
Similarly, there exists a negative integer $k^-$ such that for all $x\in [-A^\e,A^\e]$,
$$
u^\e(x) \leq v(x,p_l) - \frac{\alpha^\e}{2} \leq \tau_{k^-} U(x).
$$
Notice that $\tau_{k^\pm} U$ are also shock profiles. We now set
$$
U^\e_+ := \sup(\tau_{k^+} U ,U),\quad U^\e_- := \inf(\tau_{k^-} U , U).
$$
Since shock profiles are ordered, the functions $U^\e_\pm$ are viscous shocks, and 
$$
U^\e_-\leq u^\e\leq U^\e_+\quad\text{a.e.}
$$
Hence the lemma is proved.
\end{proof}

\subsection{An application: the convex case}

\label{ssec:convex}

This paragraph is devoted to the analysis of specific examples for which the existence of shock profiles and their stability can be proved.

\begin{lem}
Assume that for all $y\in \T$, $A(y,\cdot)$ is a convex function. Then the homogenized flux $\bar A$ is convex.

Furthermore, if $A(y,\cdot)$ is strictly convex for all $y$, then $\bar A$ is also strictly convex, and thus satisfies the Oleinik condition of Corollary \ref{cor:oleinik}.

\label{lem:convex}
\end{lem}

The convexity properties  are proved in \cite{LPV}. However, for the reader's convenience, we have reproduced the proof in  Appendix B. Oleinik's condition is an immediate consequence of the strict convexity of the flux $\bar A$.
\begin{ex}
Assume that the flux $A$ is strictly convex in its second variable, and that the assumptions of Proposition \ref{prop:ex_shock} are satisfied. Then, with the same notation as in Proposition \ref{prop:ex_shock}, we have
$$
q_l=p^+\quad\text{and}\quad q_r=p^-.
$$
Indeed, according to Corollary \ref{cor:oleinik}, we have $\{ q_l,q_r\}=\{ p^+,p^-\}$. Moreover, since the flux $A$ is strictly convex, $\p_v A(y,\cdot)$ is strictly increasing, and
$$
\mean{\p_v A(\cdot, v(\cdot,p^-))}<\mean{\p_v A(\cdot, v(\cdot,p^+))}.
$$
Proposition \ref{prop:sens} then allows us to conclude that $p^-=q_r$, $p^+=q_l$.

\end{ex}

We now prove Corollary \ref{cor:convex} (pending Proposition \ref{prop:linear}). Assume that the flux $A$ is given by
$$
A(y,p)=V(y)+ f(p),
$$
with $V$ and $f$ satisfying the assumptions of Corollary \ref{cor:convex}. The existence of solutions of equation \eqref{cell} follows immediately from Proposition \ref{prop:cellpb}; moreover, since the flux $A$ is linear at infinity, hypothesis \eqref{hyp:unif_lip} is satisfied. As a consequence, for $p>0$ sufficiently large, we have
$$
A(y,v(y,p))= V(y) + a_+ v(y,p)\quad\forall y\in \T^N,
$$
and thus $\bar A(p)=\mean{V} + a_+ p.$ Similarly, $\bar A(p)=\mean{V} - a_- p$ for $p<0$ with $|p|$ sufficiently large. These formulas entail that if $\alpha>0$ is large enough, then, setting $p_{\pm}= \pm(\alpha - \mean{V})/a_\pm$, we have $\bar A(p^-)=\bar A(p^+)=\alpha.$ Since $\bar A$ satisfies Oleinik's condition, we deduce that there exists a shock profile connecting $v(p^-)$ and $v(p^+)$. 

Additionally, if $|p|$ is large enough, then $f$ is linear, say, on the intervals $[\inf v(|p|)-1, \infty)$ and $(-\infty, \sup v(-|p|) + 1].$ Thus, for all $\xi\in[-1,1]$, we have
$$
A(y,v(y,p) + \xi)= V(y) + f(v(y,p) + \xi)= A(y,v(y,p))+ \sgn(p) a_{\sgn(p)}\xi.
$$
Hence the flux $A$ satisfies the assumption of Proposition \ref{prop:linear} for all $p$ large enough. We infer that the solutions $v(\cdot, p^\pm)$ are stable by the semi-group $S_t$ under small perturbations in $L^1_0$ which satisfy \eqref{moments}. Point \textbf{(ii)} in Corollary \ref{cor:convex} then follows from Proposition \ref{prop:stab_shock2} and the remark following it.

\section{Stability of shock profiles in one space dimension - Part I}

\label{sec:shock_stab:1}

This section is devoted to the proof of Theorem \ref{thm:stab_shock1}. Hence, throughout this section, we consider an  initial data $u_0$ which satisfies \eqref{hyp:u_0}, and such that $u_0\in U + L^1$, where $U$ is a stationary shock of equation \eqref{eq:evol_per}. Using Lemma \ref{lem:L10}, we deduce that there exists another shock profile $V$ such that $u\in V + L^1_0(\R)$. Then, using Lemma \ref{lem:up/lowshocks} together with the Contraction principle, we can restrict the analysis to the class of initial data $u_0$ such that
\be
\exists (U_-,U_+)\text{ shock profiles,}\quad U_-\leq u_0\leq U_+.
\label{in:u_0}\ee
Indeed, assume that Theorem \ref{thm:stab_shock1} holds for all $v_0\in V + L^1_0$ such that \eqref{in:u_0} is satisfied. Consider now a function $u_0\in V+ L^1_0$ satisfying \eqref{hyp:u_0}, and let $\e>0$ be arbitrary. According to Lemma \ref{lem:up/lowshocks}, there exists $u^\e_0\in V+L^1_0$ satisfying \eqref{in:u_0} and such that $\| u_0-u_0^\e\|_1\leq \e$. The $L^1$ contraction principle entails that for all $t\geq 0,$
$$
\| S_t u_0 - V\|_1\leq \|S_t u_0- S_t u_0^\e\|_1 + \| S_t u_0^\e - V\|_1\leq \e + \| S_t u_0^\e - V\|_1.
$$
Notice also that by the Contraction principle, the function $t\mapsto \| S_t u_0 - V\|_1$ is non-increasing, and thus has a finite limit as $t\to\infty$. We infer that
$$
\forall \e>0,\ \lim_{t\to\infty}\| S_t u_0 - V\|_1\leq\e,
$$
and thus $S_t u_0$ converges toward $V$ as $t\to\infty.$ 

There  remains to prove Theorem \ref{thm:stab_shock1} for initial data which satisfy \eqref{in:u_0}.
As emphasized in Section \ref{sec:results}, inequalities \eqref{hyp:u_0} or \eqref{in:u_0} should be seen as the analogues of \eqref{hyp:initiallayer} in the context of shock stability. 
The proof of Theorem \ref{thm:stab_shock1} in this case relies on arguments from dynamical systems theory, which are due to S. Osher and J. Ralston (see \cite{OR}; similar ideas are developed by D. Amadori and D. Serre in \cite{AmadoriSerre}). The aim is to prove that the $\omega$-limit set of the trajectory $S_t u_0$ is reduced to $\{ V\}$, by using a suitable Lyapunov function. Hence, we first prove that the  $\omega$-limit set, denoted by $\Omega$, is non-empty, then we state some properties of the $\omega$-limit set, and eventually we prove that $\Omega=\{ V\}.$

\vskip2mm

\noindent{\it First step. Compactness in $L^1$ of the trajectories.}

Throughout this section, we set $w(t):= S_t u_0 - V$. Notice first that by the comparison principle for equation \eqref{eq:evol_per}, inequality \eqref{in:u_0} is preserved by the semi-group $S_t$: for all $t\geq 0,$ we have
$$
U_-\leq S_t u_0\leq U_+.
$$
Hence, for all $t\geq 0,$
$$
U_--U\leq w(t)\leq U_+- U.
$$
Since $U_+-U$ and $U_- - U$ are integrable functions, the family $\{w(t)\}_{t\geq 0}$ is equi-integrable in $L^1$. Moreover, since $U_+- U $ and $U-U_- $ are bounded, it follows that $w$ is uniformly bounded in $L^\infty$. The function $w$ satisfies a linear parabolic equation of the type
$$
\p_t w + \p_y(b(t,y) w ) - \p_{yy} w=0,\quad t>0, \ y\in\R,
$$
with $b\in L^\infty([0,\infty)\times \R)$.
Theorem 10.1 in Chapter III of \cite{ladypara} then implies that there exists $\alpha>0$ such that for all $t_0\geq 1$, for all $R>0$, 
$$
\| u(t)\|_{H^{\alpha/2,\alpha}((t_0, t_0+1)\times (-R,R))}<\infty.
$$
Thus the family  $\{w(t)\}_{t\geq 0}$ is also equi-continuous in $L^1$.

Whence it follows from  the Riesz-Fr\'echet-Kolmogorov Theorem  that the family $\{w(t)\}_{t\geq 0}$ is relatively compact in $L^1(\R)$.
Thus the $\omega$-limit set
$$
 \Omega:=\left\{ W\in V + L^1(\R),\exists (t_n)_{n\in \N}, t_n\underset{n\to\infty}{\longrightarrow} \infty,\ S_{t_n}u_0\to W \text{ in }L^1(\R)  \right\}
$$
is non-empty.

\vskip1mm

\noindent{\it Second step. Properties of the $\omega$-limit set $\Omega$.}

First, $\Omega$ is forward and backward invariant by the semi-group $S_t,$ meaning that for all $t\geq 0,$
$$
S_t \Omega=\Omega.
$$
This important property is a generic one for $\omega$-limit sets. It follows immediately, thanks to parabolic regularity, that all functions in $\Omega$ are smooth: $\Omega\subset H^1_\text{loc}(\R),$ for instance. As a consequence, if $W\in \Omega$ and $w_1(t):= S_t W$, Theorem 6.1 in Chapter III of \cite{ladypara} entails that $w_1\in L^2([0,T], H^2(B_R))\cap H^1([0,T], L^2(B_R))$ for all $T,R>0.$

The second property which is important for our analysis is the LaSalle invariance principle (see \cite{lasalle}), which requires the existence of a Lyapunov function. In the case of scalar conservation laws, a classical choice for a Lyapunov function is $F[u]=\|u- V\|_1$. The Contraction principle entails that $t\mapsto F[S_t u_0]$ is non-increasing.
Thus $F$ takes a  constant value on $\Omega,$ which we denote by $C_0$.

Eventually, using the conservation of mass, we deduce that $\Omega$ is a subset of $V+ L^1_0$.

\vskip1mm

\noindent{\it Third step. Conclusion.}

 We now prove, using the parabolic structure of equation \eqref{eq:w}, that $\Omega=\{V\}$.

Let $W_0\in \Omega$ be arbitrary, and let $W(t)=S_t (W_0)$. Notice that $W(t)\in \Omega$ for all $t\geq 0$, according to the previous step. Moreover, $W-V$ satisfies
$$
\p_t (W-V) + \p_y \left( A(y, W) - A(y, V) \right) - \p_{yy} (W-V)=0.
$$

Multiplying the above equation by $\sgn(W-V)$, we obtain
$$
\p_t |W-V| + \p_y \left[\sgn(W-V) \left(A(y,W(t)) - A(y,V)\right)\right] - \sgn(W-V) \p_{yy} (W-V)=0.
$$
Let $\phi$ be a cut-off function, i.e. $\phi\in\mathcal C^\infty_0(\R)$, $\phi\geq 0$ and $\phi\equiv 1$ in a neighbourhood of zero. For $R>0$, we set $\phi_R:= \phi(\cdot/R).$
We now multiply the above equality by $\phi_R$ and integrate on $[t,t']\times\R$. Recalling that $\int_\R |W(t)-V|=C_0$ for all $t$, we infer that for all $t'>t\geq 0$, there exists a function $\e_{t,t'}:[0,\infty)\to[0,\infty)$ such that $\lim_{R\to\infty} \e_{t,t'}(R)=0$ and
$$
\left|\int_t^{t'}\int_{\R}\sgn(W(s)-V) \p_{yy}(W(s,y)-V(y))\phi_R(y)\:ds\:dy\right|\leq \e_{t,t'} (R).
$$
Thus, using a slightly modified version of Lemma \ref{lem:tech} in the Appendix, we infer that
$$
\sgn(W(s) - V) \p_{yy}(W(s)-V)=\p_{yy}|(W(s)-V)|
$$
almost everywhere and in the sense of distributions. Consequently, the function $|W-V|$ is a non-negative solution of  a parabolic equation of the type
$$
\p_t |W-v| + \p_y (b(t,y)|W-V|) -\p_{yy}|W-V|=0,
$$
with $b\in L^\infty([0,\infty\R)$. We now conclude thanks to Harnack's inequality  (see \cite{evans}): let $x_0\in\R$ be arbitrary, and let $K$ be any compact set in $\R$ such that $x_0\in K$. Then there exists a constant $C_K$ such that
$$
|(W_0- V)(x_0)|\leq \sup_{x\in K}|(W_0- V)(x)| \leq C_K \inf_{x\in K}|(W_{|s=1}-V) (x)|.
$$
Now, $(W_{|s=1}-V)\in L^1_0\cap H^1_{\text{loc}}(\R)$, and thus there exists  $x_1\in\R$ such that
$$
W(1,x_1) - V(x_1)=0.
$$
Choose $K$ such that $x_1\in K$. Then
 $W_0-V$ vanishes uniformly on $K$, and in particular, $(W_0-V)(x_0)=0$. Since $x_0$ was chosen arbitrarily, we deduce that $W_0=V$. Hence $\Omega=\{V\},$ and Theorem \ref{thm:stab_shock1} is proved.

\section{Uniform in time {\it a priori} bounds for viscous scalar conservation laws}

\label{sec:bounds}

This section is devoted to the proof of Proposition \ref{prop:bounds}. As far as possible, we will treat both models simultaneously. We set
$$
w(t) := S_t u_0 - U_0, \quad t\geq 0.
$$
The function $w$ satisfies the following equation
\be\label{eq:w}
\p_t w(t,y) + \dv_y B(y,w(t,y)) - \Delta_y w(t,y)=0,\quad t>0, \ y\in Q,
\ee
where
$$
B(y,w)=A(y,U_0(y)+ w ) - A(y,U_0(y)),\quad y\in Q,\ w\in\R.
$$

Due to the Contraction principle recalled in Section \ref{sec:results}, it is known that $w$ is bounded in $L^\infty([0,\infty), L^1(Q))$, and 
\be
\forall t\in \R_+,\quad \| w(t)\|_{L^1}\leq \| u_0 - U_0\|_{L^1}.\label{L1_bound}
\ee
The idea of this section is to use this uniform $L^1$ bound in order to derive uniform $L^p$ bounds on $w$ for all $p\in[1,\infty]$. To that end, we proceed by induction on the exponent $p$. The first step is dedicated to the derivation of a differential inequality relating the derivative of the $L^p$ norm to a viscous dissipation term. The calculations are very similar to those developed in \cite{homogpara} to derive \textit{a priori} bounds for solutions of equation \eqref{cell}. Then, we use Poincar\'e inequalities to control the $L^p$ norm by the dissipation. Eventually, we conclude thanks to a Gronwall type argument.

\vskip2mm

\noindent\textit{Preliminary for the whole space case.}

We begin by recalling some regularity results about the solutions of equation \eqref{eq:evol_per} in the case $Q=\R$. According to the papers by Kru\v{z}kov \cite{kruzkhov1,kruzkhov2}, it is kown that $w\in L^\infty_\text{loc}([0,\infty), L^\infty (Q))$. As a consequence, $w\in L^\infty_\text{loc}([0,\infty), L^p (Q))$ for all $p$. 

Then, multiplying \eqref{eq:w} by $w\chi$ where $\chi\in \mathcal C^\infty_0(\R)$ is an arbitrary non-negative cut-off function, and integrating in space and time, it is easily proved that for all $T>0$,  $w$ satisfies an inequality of the type
$$
\int_0^T \int_{\R} |\p_y w(s,y)|^2 \chi(y)\:dy\:ds\leq C_T,
$$
where the constant $C_T$ depends on $T$, $\|w \|_{L^\infty([0,T]\times \R)}$ and $\| w_{t=0}\|_1$, but not on $\chi$. We deduce that $\p_y w \in L^2_{\text{loc}}([0,\infty), L^2(\R))$.

\vskip2mm

\noindent {\it First step. A differential inequality.}

In this step, we treat the periodic and the full space models simultaneously; our goal is to prove an inequality of the type
$$
\frac{d}{dt}\int |w|^{q+1} + c_q\int \left|\nabla |w|^{\frac{q+1}{2}}\right|^2  \leq C_q \left(\int |w|^{q+n} + \int |w|^{q+1}\right),
$$
where $q\geq 1$ is arbitrary, $n$ is the exponent appearing in \eqref{hyp:A1}, and the constants $c_q$ and $C_q$ depend on $q$, $n$,  $N$, and $\| U_0\|_{W^{1,\infty}}$.

To that end, we take $q\geq 1$,  multiply \eqref{eq:w} by $w|w|^{q-1}$ and integrate on $Q$; we obtain
\begin{multline*}
 \frac{1}{q+1}\frac{d}{dt}\int_Q |w|^{q+1}+ q \int_Q |\nabla w|^2 |w|^{q-1}=\\
=q\int_Q \nabla_y w(t,y)\cdot B(y,w(t,y))|w(t,y)|^{q-1}\:dy.
\end{multline*}

Notice that all terms are well-defined thanks to the preliminary step.

For $(y,w)\in Q\times \R$, set
$$
b_q(y,w)=q\int_0^w B(y,w')|w'|^{q-1}dw';
$$
then
\begin{eqnarray*}
 &&-q\int_{\T} \nabla_y w(t,y)\cdot B(y,w(t,y))|w(t,y)|^{q-1}\:dy\\ &=& \int_{\T} \left[- \dv_y\left(b_q(y,w(t,y))  \right)+(\dv_y b_q)(y,w(t,y))\right] \:dy\\
&=&q\int_{\T} \int_0^{w(t,y)}(\dv_y B)(y,w')|w'|^{q-1}\:dw'.
\end{eqnarray*}
Thus, we now compute, for $(y,w')\in Q\times \R,$ 
\begin{eqnarray*}
\dv_y  B(y,w')&=&\dv_y \left[ A(y,U_0(y)+ w') -A(y,U_0(y))  \right]\\
&=&(\dv_y A)(y,U_0(y) + w') - (\dv_y A)(y,U_0(y) )\\
&+& \nabla_y U_0\cdot \left[(\p_v A)(y,U_0(y) + w')-(\p_v A)(y,U_0(y) )\right].
\end{eqnarray*}
Consequently, according to hypothesis \eqref{hyp:A1}, we deduce that there exists a positive constant $C$ depending only on $\| U_0\|_{W^{1,\infty}}$ and $q$ such that
$$
\left|q\int \nabla_y w(t,y)\cdot B(y,w(t,y))|w(t,y)|^{q-1}\:dy\right|\leq C \left(\int |w(t)|^{q+1} + \int |w(t)|^{q+n}\right).
$$
Eventually, we infer that for all $q\geq 1$, there exist positive constants $c_q$, $C_q$ such that for all $t>0$,
\be\label{in:energy-per}
\frac{d}{dt}\int |w(t)|^{q+1} + c_q \int \left| \nabla |w(t)|^{\frac{q+1}{2}} \right|^2\leq C_q\left(\int |w(t)|^{q+1} + \int |w(t)|^{q+n}\right).
\ee

\vskip2mm
\noindent{ \it Second step. Control of $L^p$ norms by the dissipation term (Poincar\'e inequalities).}

In this step, we treat the periodic case and the whole space case separately, and we begin with the periodic case.

First, remember that for all $p \in(1,\infty)$ such that $\frac{1}{p}\geq \frac{1}{2}- \frac{1}{N}$, there exists a positive constant $C_p$ such that for all $\phi\in \HP$,
\be
\left\| \phi - \mean{\phi}\right\|_p\leq C_p \left\| \nabla \phi \right\|_2.
\ee
Taking $\phi= |w|^{\frac{q+1}{2}}$, we deduce that
$$
\| w\|_r\leq C_r\left(\left\| \nabla |w|^{\frac{q+1}{2}} \right\|^{\frac{2}{q+1}}_2 + \| w\|_{\frac{q+1}{2}}\right),
$$
where $r\in(1,\infty)$ is such that 
\be\label{cond:r1}
\frac{1}{r}\geq \frac{1}{q+1}- \frac{2}{N(q+1)}.
\ee

Now, the idea is to interpolate the $L^{n+q}$ and the $L^{q+1}$ norms in the right-hand side of inequality \eqref{in:energy-per} between $L^1$ and $L^r$, where $r$ satisfies the  constraint above. It can be easily checked that when $n<N+2)/N$,  we have
$$
\frac{1}{n+q}> \frac{1}{q+1}- \frac{2}{N(q+1)};
$$
hence the interpolation is always possible, and we have
$$
\begin{aligned}
\|w\|_{q+1}\leq \|w\|_1^{1-\alpha} \|w\|_r^\alpha,\\
 \|w\|_{q+n}\leq \|w\|_1^{1-\beta} \|w\|_r^\beta,
\end{aligned}
$$
where
$$
\frac{1}{q+1}= 1-\alpha + \frac{\alpha}{r},\quad \frac{1}{q+n}= 1-\beta + \frac{\beta}{r}.
$$
Gathering all inequalities, we infer that
\begin{eqnarray*}
&& \frac{d}{dt}\|w\|_{q+1}^{q+1} + C_q \left\| \nabla |w|^{\frac{q+1}{2}} \right\|_2^2\\
&\leq & C \|w\|_1^{(q+1)(1-\alpha)}\left\| \nabla |w|^{\frac{q+1}{2}} \right\|^{2\alpha}_2+C  \|w\|_1^{(q+1)(1-\alpha)}\| w\|_{\frac{q+1}{2}}^{\alpha(q+1)}\\
&+& C \|w\|_1^{(q+n)(1-\beta)}\left\| \nabla |w|^{\frac{q+1}{2}} \right\|^{\frac{2\beta(q+n)}{q+1}}_2+ C \|w\|_1^{(q+n)(1-\beta)}\| w\|_{\frac{q+1}{2}}^{\beta(q+n)}.
\end{eqnarray*}
Remember that the $L^1$ norm is bounded. For the time being, we leave aside the $L^{\frac{q+1}{2}}$ norms of the right-hand side: those will be treated in the very last step. In order to control the right-hand side by the dissipation term in the left-hand side, it suffices to find $r$ (and thus $\alpha$ and $\beta$) such that
\be\label{cond:r2}
2\alpha<2,\quad\frac{2\beta(q+n)}{q+1}<2.
\ee
Remembering the definition of $\beta$, we deduce that we have to find $r\in (q+1,\infty)$ satisfying the two inequalities
$$
\begin{aligned}
1- \frac{1}{r}>\frac{q+ n-1}{q+1} ,\\
\frac{1}{r}\geq \frac{1}{q+1}- \frac{2}{N(q+1)}.
\end{aligned}
$$
This is possible if and only if the couple $(n,q)$ satisfies
$$
\left\{ \begin{array}{l}
\ds \frac{q+ n-1}{q+1}<1,\\~\\
\ds\frac{1}{q+1}- \frac{2}{N(q+1)}< 1-\frac{q+ n-1}{q+1}
\end{array}
 \right.
$$
which amounts to the condition $n<\min(2,(N+2)/N)$. In the case when $N=1$, this yields $n<2$, which is more restrictive than the assumption of Proposition \ref{prop:bounds} ($n<3$). However, when $N=1$, the same arguments as in the whole space case can be used (see below), and lead to $n<3$. Thus, under the hypotheses of Theorem \ref{thm:lgtime}, for all $q\geq 1$, we may find $r>\max(q+1,q+n)$ such that conditions \eqref{cond:r1}, \eqref{cond:r2} are fulfilled. The Cauchy-Schwarz inequality then implies that
\be\label{in:pre-gronwall}
\frac{d}{dt}\int |w(t)|^{q+1} + C_1 \int \left| \nabla |w(t)|^{\frac{q+1}{2}} \right|^2\leq C_2\left(  \| w(t)\|_{\frac{q+1}{2}}^{p_1} + \| w(t)\|_{\frac{q+1}{2}}^{p_2}\right),
\ee
where the constant  $C_2$ depends on $\| u_0-U_0\|_1,$ and the exponents $p_1,p_2$ on $n$, $q$ and $N$.
According to the Poincar\'e-Wirtinger inequality, we have
\begin{eqnarray*}
  \left\| \nabla |w|^{\frac{q+1}{2}} \right\|_2^2 &\geq & c \left\| |w|^{\frac{q+1}{2}}- \mean{|w|^{\frac{q+1}{2}}}\right\|_2^2\\
&=&c\left(\left\| |w|^{\frac{q+1}{2}} \right\|_2^2 - \mean{|w|^{\frac{q+1}{2}}}^2\right)\\
&=&c \left( \int |w|^{q+1} - \| w\|_{\frac{q+1}{2}}^{q+1} \right).
\end{eqnarray*}
Eventually, we deduce that for all $q\geq 1$, there exists constants $C_1, C_2,p_1,p_2$ such that
\be\label{in:gronwall}
\frac{d}{dt}\int |w(t)|^{q+1} + C_1 \int |w(t)|^{q+1}\leq C_2\left(  \| w(t)\|_{\frac{q+1}{2}}^{p_1} + \| w(t)\|_{\frac{q+1}{2}}^{p_2}\right).
\ee

\vskip2mm

Let us now treat the one-dimensional model set in the whole space. In dimension one, the $H^1$ and $L^1$ norms control the $L^\infty$ norm. Hence we now interpolate the two integrals in the 
right-hand side of \eqref{in:energy-per} between $L^{\frac{q+1}{2}}$ and $L^\infty$:
$$
\begin{aligned}
\int |w|^{q+1}\leq \| w\|_{\frac{q+1}{2}}^\frac{q+1}{2} \| w\|_{\infty}^\frac{q+1}{2},\\
 \int |w|^{q+n}\leq \| w\|_{\frac{q+1}{2}}^\frac{q+1}{2} \| w\|_{\infty}^\frac{q+2n-1}{2}.
\end{aligned}
$$
We use the following Poincar\'e inequality, which involves the dissipation term in the right-hand side of \eqref{in:energy-per} (the proof of this inequality is classical and left to the reader: we refer to \cite{GT} for the proof of similar inequalities): there exists a constant $C_q$, depending only on $q$, such that for all $w\in L^1\cap L^\infty\cap H^1(\R)$\footnote{If $w\in L^\infty\cap H^1(\T)$, the corresponding inequality is
$$
\|w\|_{L^\infty(\T)}\leq C_q\|w \|_\frac{q+1}{2}^{1/3}\left( \left( \int \left| \p_y |w|^\frac{q+1}{2} \right|^2\right)^\frac{2}{3(q+1)} +\|w \|_\frac{q+1}{2}^{2/3} \right).
$$
}
$$
\| w\|_{\infty}\leq C_q \|w \|_\frac{q+1}{2}^{1/3}\left( \int \left| \p_y |w|^\frac{q+1}{2} \right|^2\right)^\frac{2}{3(q+1)}.
$$
Consequently, there exist positive constants $C,p$ such that for all $w\in L^1\cap L^\infty\cap H^1(\R)$,
$$
\int |w|^{q+n}\leq C \| w\|_\frac{q+1}{2}^p\left( \int \left| \p_y |w|^\frac{q+1}{2} \right|^2\right)^\frac{q+2n-1}{3(q+1)}.
$$
Hence, in order that the dissipation term controls the right-hand side of \eqref{in:energy-per}, the exponent $n$ should satisfy
$$
\frac{q+2n-1}{3(q+1)}<1\quad\forall q\geq 1,
$$
which leads to the condition $n<3$. Using Young's inequality, we conclude that \eqref{in:pre-gronwall} is satisfied. Moreover, the Poincar\'e inequality used above entails that for all $\lambda>0$,
\begin{eqnarray*}
 \int |w|^{q+1}&\leq& C \|w\|_\frac{q+1}{2}^\frac{2(q+1)}{3} \left( \int \left| \p_y |w|^\frac{q+1}{2} \right|^2\right)^{1/3}\\
&\leq & \frac{C}{\lambda^2}  \|w\|_\frac{q+1}{2}^{q+1} + \lambda \int \left| \p_y |w|^\frac{q+1}{2} \right|^2.
\end{eqnarray*}
Eventually, we deduce that inequality \eqref{in:gronwall} is also satisfied in the whole space case.

\vskip2mm

\noindent{\it  Third step. Uniform bounds in $L^q$ for all $q<\infty.$}

We now conclude thanks to Gronwall's lemma, using an inductive argument. Notice indeed that inequality \eqref{in:gronwall} implies that for all $q\geq 1$,
\be\label{rec}
w\in L^\infty([0,\infty), L^{q}(Q)) \Rightarrow w\in L^\infty([0,\infty), L^{2q}(Q)).
\ee
Indeed, assume that $w\in L^\infty([0,\infty), L^q)$ for some $q\geq 1$. According to \eqref{in:gronwall}, we have
$$
\frac{d}{dt}\int |w(t)|^{2q} + C_1 \int |w(t)|^{2q}\leq C_2,
$$
where the constant $C_2$ depends on $\| U_0\|_{W^{1,\infty}}$ and on $\| w\|_{L^\infty([0,\infty), L^1)}$, so that, using Gronwall's lemma,
$$
\int |w(t)|^{2q} \leq e^{-C_1 t}\int |w_{|t=0}|^{2q} + \frac{C_2}{C_1}(1- e^{-C_1t})\leq C.
$$
Thus $w\in L^\infty([0,\infty), L^{2q})$ and \eqref{rec} is proved. Since $w\in L^\infty([0,\infty), L^{1})$, we deduce that $w\in L^\infty([0,\infty), L^{q})$ for all $q\in[1,\infty)$.

\vskip2mm

{\it \noindent Fourth step. Uniform bounds in $L^\infty$ and $W^{1,p}.$}

We now derive some $L^\infty$ bounds thanks to parabolic regularity results. First, notice that in equation \eqref{eq:w}, the flux $B$ can be written as
$$
B(y,w(t,y))=b(t,y) w,
$$
where
$$
b(t,y)= \int_0^1 a(y,v_0(y) + \tau w(t,y))\:d\tau.
$$
According to the previous steps, $b(t,y)\in L^\infty([0,\infty), L^q_\text{loc}(Q))$ for all $q>0$; in particular, in the whole space case, for all $q>1$ there exists a constant $C_q$ such that for all $y_0\in\R$,
$$
\sup_{t\geq 0} \| b(t)\|_{L^q(y_0-2, y_0+2)}\leq C_q.
$$
We now use Theorem 8.1 in Chapter III of \cite{ladypara}: we have, for all $y_0\in Q$, for all $t_0\geq 1$,
$$
|w(t_0,y_0)|\leq C \left(\| w\|_{L^2(Q_{t_0,y_0})}, \| b\|_{L^q(Q_{t_0,y_0})}\right),
$$
where $Q_{t_0,y_0}:=(t_0-1,t_0+1)\times (y_0-1,y_0+1)$ and $q$ is some parameter chosen sufficiently large. The right-hand side is bounded uniformly in $y_0$ and $t_0$ by a positive constant $C$, and we infer that for all $y_0\in Q, t_0\geq 1,$
$$
|w(t_0,y_0)|\leq C.
$$
Thus $w\in L^\infty([0,\infty)\times Q)$.
Using Theorem 10.1 in Chapter III of \cite{ladypara}, we also deduce that there exists $\alpha >0$ and a constant $C>0$ such that for all $t_0\geq 1,$ for all $x_0\in Q$, 
$$
\left\| w\right\|_{H^{\alpha/2,\alpha}((t_0,t_0+1)\times (x_0-1,x_0+1))} \leq C.
$$
As a consequence, we obtain
$$
\| w\|_{L^\infty([1,\infty), \mathcal C^\alpha(Q))}\leq C.
$$

\section{Long time behaviour of solutions for the periodic model}

\label{sec:per-stab}

Throughout this section, we assume that $Q=\T^N$, and we consider a solution $u(t)=S_t u_0$ of equation \eqref{eq:evol_per} ($t\geq 0$).  Our goal is to prove, under the assumptions of Theorem \ref{thm:lgtime}, that $u(t) - v(\mean{u_0})$ vanishes in $L^\infty$ as $t\to\infty$. The idea is to prove in a first step the convergence for initial data which are bounded from above or from below by a solution of equation \eqref{cell}, and then to extend this result to arbitrary initial data thanks to the $L^\infty$ bounds proved in the previous section (see Proposition \ref{prop:bounds}).
We thus begin with the following Proposition:
\begin{prop}
Let $u_0\in L^\infty(\T^N)$ such that 
\be
\exists p_0\in \R,\quad u_0(y)\leq v(y,p_0) \quad\text{for a.e. } y\in \T^N.
\ee

Let $u(t)=S_t u_0$ for $t\geq 0$. Then, as $t\to\infty,$
$$
u(t)\to v(\cdot, \mean{u_0})\quad \text{in } L^\infty(\T^N).
$$
\label{prop:conv_sup}
\end{prop}

Of course, the same result holds when the upper-bound is replaced by a lower-bound:

\begin{corol}
 Let $u_0\in L^\infty(\T^N)$ such that 
\be
\exists p_0\in \R,\quad u_0(y)\geq v(y,p_0) \quad\text{for a.e. } y\in \T^N.
\ee

Let $u(t)=S_t u_0$ for $t\geq 0$. Then, as $t\to\infty,$
$$
u(t)\to v(\cdot, \mean{u_0})\quad \text{in } L^\infty(\T^N).
$$

\label{cor:conv_inf}
\end{corol}

\begin{proof}[Proof of Proposition \ref{prop:conv_sup}]
According to the previous section (see Proposition \ref{prop:bounds}), 
$$
\sup_{t\geq 0}\|u(t) \|_{L^\infty(\T^N)}<+\infty.
$$ 
Additionally, the Comparison principle yields
$$
u(t,y)\leq v(y,p_0)\quad\forall t>0, \ \forall y\in \T^N.
$$
From now on, the proof is very close to that in \cite{initiallayer}, Section 2: we   recall the main steps for the reader's convenience. Set
$$
\begin{aligned}
U(t,y):=\sup_{t'\geq t} u( t',y),\quad t\geq 0, y\in \T^N,\\
p^*(t):=\inf \left\{ p\in \R, v(y,p)\geq U(t,y)\text{ for a.e. }y\in \T^N \right\}, t\geq 0.
\end{aligned}
$$
Then $U$ belongs to $L^\infty([0,\infty)\times \T^N)$ (since $u$ is uniformly bounded in time), and $U$ is clearly a non-increasing function. Moreover, $U$ satisfies
$$
U(t,y)\leq v(y,p_0)\quad\forall t>0, \ \forall y\in \T^N.
$$
As a consequence, $p^*(t)$ is bounded from above by $p_0$, and $p^*$ is a non-increasing function. Moreover, $p^*$ is bounded from below, since for almost every $y\in\T^N$,
$$
v(y,p^*(t))\geq U(t,y) \geq -\| u\|_{L^\infty([0,\infty)\times \T^N)},
$$
and thus
$$
\forall t\geq 0,\quad p^*(t)=\mean{v(\cdot,p^*(t))}\geq -\| u\|_{L^\infty([0,\infty)\times \T^N)}.
$$
Hence $p^*$ is a bounded decreasing function, and thus $p^*(t)$ has a finite limit, which we denote by $\bar p^*$, as $t\to \infty.$

The idea is to prove that $u(t)- v(\cdot, \bar p^*)$ converges towards zero as $t\to \infty.$  Let $\e>0$ be arbitrary. We first choose $t_0>0$ such that 
$$
\|v(p^*(t))- v(\bar p^*)\|_\infty\leq \e\quad\forall t\geq t_0,
$$
and then we pick $p<\bar p^*$ and $y_0\in \T^N$ such that
$$
v(y_0, \bar p^*)-\e \leq v(y_0,p)\leq U(t_0+1,y_0)\leq v(y_0,p^*(t_0+1)) \leq v(y_0, \bar p^*)+\e.
$$
Now, choose $t_1\geq t_0+1$ such that 
$$
 U(t_0+1,y_0)-\e \leq u(t_1,y_0)\leq U(t_0+1,y_0).
$$
By construction, the function 
$$
 V :(s,y)\in\left( -1,1 \right) \times \T^N \mapsto v(y,p^*(t_0)) - u(t_1+ s ,y) 
$$
is a non-negative solution of a linear diffusion equation of the type
$$
\p_t V + \dv_y (b V)- \Delta_y V=0
$$
for some vector field $b\in L^\infty([-1,1]\times \T^N)^N$. Hence by Harnack's inequality, there exists a constant $C$ such that
$$
\sup_{y\in \T^N} V\left(-\frac{1}{2},y\right)\leq C \inf_{y\in \T^N}V\left(0,y\right)\leq C \e.
$$
Thus, there exists a sequence  of positive numbers $(t_n)$ such that $\lim_{n\to\infty}t_n=+\infty$ and such that $u(t_n)$ converges towards $v(\bar p^*)$ in $L^\infty$. The $L^1$ contraction principle, together with parabolic regularity results, entails that the whole family $u(t)$ converges. Eventually, we obtain that $\bar p^*= \mean{u_0}$ by conservation of mass.

\end{proof}

The core of the proof of Theorem \ref{thm:lgtime} then lies in the following argument: if $u_0\in L^\infty$ is arbitrary, we set 
$$\begin{aligned}
   \tilde u_0:= \inf(u_0 , v( p)),\\
\tilde u:= S_t \tilde u_0.
  \end{aligned}
$$ 
The value of parameter $p$ above is irrelevant. One can choose for instance $p=0$, or $p= \mean{w_0}$.

The function $\tilde u_0$ obviously satisfies the assumptions of Proposition \ref{prop:conv_sup}. Hence as $t\to\infty$,
$$
\tilde u(t)\to v\left(\mean{\tilde u_0} \right)\quad\text{in }L^\infty,
$$
and thus there exists a positive time $t_0$ such that for $t\geq t_0,$ for all $y\in \T^N$,
$$
\tilde u(t,y)\geq v\left(y,  \mean{\tilde u_0} -1\right).
$$
On the other hand, notice that $\tilde u_0\leq u_0$ by definition, and thus by the comparison principle,
$$
\tilde u(t)\leq u(t)\quad \forall t.
$$
Hence, for $t\geq t_0,$
$$
 u(t)\geq v\left( \mean{\tilde u_0} -1\right).
$$
In particular, $u(t_0)$ satisfies the assumptions of Corollary \ref{cor:conv_inf}, and thus, as $t\to \infty$,
$$
S_t u(t_0) \to v\left( \mean{u(t_0)} \right).
$$
 Since
$$
u(t)=S_{t-t_0} u(t_0)
$$
and $\mean{u(t_0)}=\mean{u_0}$ by the Conservation property, we deduce eventually that 
$$
u(t)\to v\left( \mean{u_0} \right)\quad \text{as }t\to \infty.
$$
Thus Theorem \ref{thm:lgtime} is proved.

\section{Stability of shock profiles in one space dimension - Part II}

\label{sec:shock_stab:2}

This section is devoted to the proof of additional results on  shock stability in the whole space case. We start by proving Proposition \ref{prop:stab_shock2}, and then we prove that the conclusion of Proposition \ref{prop:stab_shock2} still holds when \textbf{(H)} is replaced by \textbf{(H')}. We have not been able to prove that \textbf{(H')} is satisfied for arbitrary fluxes. Thus, at the end of this section, we prove Proposition \ref{prop:linear} and thereby provide explicit examples for which \textbf{(H')} is satisfied. We also explain which difficulties are encountered when trying to prove \textbf{(H')}.

\vskip2mm

We start by introducing some notation. Following \cite{SerreHandbook}, we denote by $\cG$ the set of shock profiles connecting $v(p_l)$ to $v(p_r)$, and we set
$$\begin{aligned}
   \cA:=\left\{ u\in L^\infty_\text{loc}(\R), \ \exists U\in\cG,\  u\in U + L^1(\R)\right\},\\
\cA_0:=\left\{ u\in \cA,\ v(\min(p_l, p_r))\leq u\leq v(\max(p_l, p_r))\right\}.
  \end{aligned}
$$
Our goal is to prove that for all $u_0\in\cA$,
$$
d(S_t u_0, \cG)=0,
$$
where $d(u,A)$ denotes the $L^1$ distance from $u$ to a set $A$. Notice that the Contraction principle easily entails that the function $t\mapsto d(S_t u_0, \cG) $ is decreasing. Hence, its limit as $t\to\infty$ exists; for all $u_0\in \cA$, set
$$
\ell_0 (u_0):=\lim_{t\to\infty} d(S_t u_0, \cG).
$$
Theorem \ref{thm:stab_shock1} states that $\ell_0(u)=0$ for all $u\in \cA_0.$ Moreover, it follows from the Contraction principle  that $\ell_0(u_0)$ is a contraction, i.e.
$$
\left| \ell_0(u) -\ell_0(v) \right|\leq \|u-v \|_{L^1} \quad\forall u,v\in \cA.
$$
Additionally, for all $t\geq 0$ and for all $u\in \cA$,
$$
\ell_0(u)=\ell_0(S_t u).
$$
Similarly, we define, for all $u_0\in\cA$,
$$
\ell_1 (u_0):=\lim_{t\to\infty} d(S_t u_0, \cA_0).
$$
The function $\ell_1$ is well-defined: indeed, the Comparison property entails that $\cA_0$ is stable by the semi-group $S_t$. Consequently, by the Contraction principle, the function $t\mapsto d(S_t u_0, \cA_0)$ is decreasing and non-negative, and thus has a finite limit as $t\to\infty.$ Moreover, the functional $\ell_1$ enjoys the same properties as $\ell_0$: $\ell_1$ is a contraction on $\cA$ and $\ell_1(u)=\ell_1(S_t u)$ for all $t\geq 0$. Eventually, since $\cG\subset \cA_0$, we deduce that
$$
\ell_1(u)\leq \ell_0(u)\quad\forall u\in\cA.
$$

\subsection{Proof of Proposition \ref{prop:stab_shock2}}
We now tackle the proof of Proposition \ref{prop:stab_shock2}, which is very similar to \cite{SerreHandbook}, paragraph 3.5. 
Let $u_0\in \cA$ be arbitrary. For all $v\in \cA_0$, we have
$$
\ell_0(u_0)\leq \ell_0(v) + \|u_0-v\|_1 \leq \|u_0-v\|_1.
$$Thus for all $u_0\in \cA$,
$$
\ell_0(u_0)\leq d(u_0, \cA_0).
$$
Replacing $u_0$ by $S_t u_0$ in the previous inequality, we infer that for all $u_0\in\cA_0$,
$$
\ell_0(u_0)\leq \lim_{t\to\infty }d(S_t u_0, \cA_0)=\ell_1(u_0).
$$
Thus $\ell_0$ and $\ell_1$ take the same values on $\cA$, and
it suffices to prove that
\be
\ell_1(u_0)=\lim_{t\to\infty} d(S_t u_0, \cA_0)=0.\label{dA_0}
\ee
Notice that if $u\in\cA$, then, with $p^+=\max(p_l,p_r)$, $p^-=\min(p_l,p_r)$,
$$
d(u, \cA_0)= \left\| \left(u-v(p^+)\right)_+\right\|_1 + \left\| \left(u-v(p^-)\right)_-\right\|_1.
$$

We now prove that assumption \textbf{(H)} implies \eqref{dA_0}.
According to Lemma \ref{lem:L10}, there exists a shock profile $U$ such that $u\in U + L^1_0(\R)$.
We now define functions $a^+, a^-$ in $v(p^+) + L^1_0$ and $v(p^-) + L^1_0$ respectively, such that
$$
a^-(y)\leq u_0(y)\leq a^+(y).
$$
Let us explain for instance the construction of $a^+$. If $u_0(y)> v(y,p^+),$ we set
$$
a^+(y)=u_0(y).
$$
On the other hand, since $u\in U + L^1$ and $U$ is asymptotic to $v(p^+), v(p^-)$, we have
$$
\int_\R (v(y,p^+)- u_0(y))_+\:dy \geq \int (v(p^+) -U) - \| u_0-U\|_1 = +\infty.
$$
Hence there is enough room, between the graphs of $v(y,p^+)$ and  $u_0(y)$ (restricted to the set where $u_0(y)\leq v(y,p^+)$), to insert a function $b^+$ such that 
$$\begin{aligned}
u_0(y)\leq v(y,p^+)\Rightarrow u_0(y)\leq b^+(y) \leq v(y,p^+),\\
\int_{\R}\mathbf 1_{u_0\leq v(y,p^+)}(v(y,p^+)- b^+(y))\:dy =\int_{\R}\mathbf 1_{u_0>v(y,p^+)}(u_0(y)-v(y,p^+))\:dy.
  \end{aligned}
$$
On the set where $u_0(y)\leq v(y,p^+)$, we define $a^+(y)=b^+(y)$. It is obvious that the function $a^+$ belongs to $v(p^+)+ L^1_0$ and that $u_0\leq a^+$. The function $a^-$ is defined in a similar fashion.
Thanks to the comparison principle, we have
$$
S_t a^- \leq S_t u_0 \leq S_t a^+\quad\forall t\geq 0.
$$
Consequently,
\be\label{maj:A(U)}
d(S_t u_0, \cA_0)\leq \left\|S_t a^+- v(p^+) \right\|_{L^1} + \left\|S_t a^-- v(p^-) \right\|_{L^1}.
\ee
From the above inequality, it is clear that \textbf{(H)} entails \eqref{dA_0}: if $S_t a^\pm -v(p^\pm)$ vanish in $L^1$, then $\ell_1(u_0)=0.$ In other words, the stability of shock profiles follows from to the stability of solutions of equation \eqref{cell} in $L^1_0$. Thus, we now focus on the case when merely \textbf{(H')} is satisfied.

Let $\delta>0$. If $u_0\in\cA$ is such that 
$$
\|(u_0-v(p^+))_+\|_1\leq \delta,\quad \|(u_0-v(p^-))_-\|_1\leq \delta,
$$
then by construction
$$
\|a_+- v(p^+)\|_1\leq 2\delta,\quad\|a_-- v(p^-)\|_1\leq 2\delta.
$$
And according to {\bf(H')}, there exists $\delta_0>0$ such that if $\delta\leq \delta_0$, then
$$
\lim_{t\to\infty}\| S_t a_\pm - v(p^\pm)\|_1=0,
$$
and thus the right-hand side of \eqref{maj:A(U)} vanishes as $t\to\infty$. Thus $\ell_1(u_0)=0$.

Hence we now focus on the case where
$$
\|(u_0-v(p^+))_+\|_1\geq \delta_0\text{ or }\|(u_0-v(p^-))_-\|_1\geq \delta_0.
$$
We then define the function 
$$
\bar u_0(y):=\left\{ 
			\begin{array}{ll}
			v(y,p^+ ) + \alpha_+ (u_0-v(p^+))&\text{ if }u_0 (y)>v(y,p^+ ),\\
			u_0(y)&\text{ if } v(y,p^-)\leq u_0(y)\leq v(y,p^+) ,\\
			v(y,p^- ) + \alpha_- (u_0-v(p^-))&\text{ if }u_0 (y)<v(y,p^-),
\end{array}
\right.
$$
where
$$
\alpha_\pm= \left\{ 
			\begin{array}{ll}\frac{\| (u_0-v(p^\pm))_\pm\|_1}{\delta_0}&\text{ if } \| (u_0-v(p^\pm))_\pm\|_1>\delta_0,\\
			 0&\text{ else.}
			\end{array}\right.
$$
Since $\bar u_0-u_0\in L^1(\R)$, $\bar u_0\in\cA$. Moreover,
\begin{eqnarray*}
\|\bar u_0-u_0 \|_1&=& (1-\alpha_+) \|(u_0-v(p^+))_+\|_1+ (1-\alpha_-) \|(u_0-v(p^-))_-\|_1\\
&\leq & d(u_0,\cA_0) - \delta_0.
\end{eqnarray*}
Notice that $\ell_1(\bu_0)=0$. Since $\ell_1$ is a contraction, we have
\be\label{in:l1}
\ell_1(u_0)\leq \ell_1(\bu_0) + \| u_0-\bu_0\|_1 \leq  d(u_0,\cA_0) - \delta_0.
\ee

We now argue by contradiction. Assume that for all $t\geq 0$,
$$
\|(S_tu_0-v(p^+))_+\|_1\geq \delta_0\text{ or }\|(S_tu_0-v(p^-))_-\|_1\geq \delta_0.
$$
Then we may replace $u_0$ by $S_t u_0$, for $t\geq 0$ arbitrary, in inequality \eqref{in:l1}. We obtain
$$
\ell_1(u_0)=\ell_1(S_t u_0)\leq d(S_t u_0,\cA_0) - \delta_0.
$$
Passing to the limit as $t\to\infty$, we infer
$$
\ell_1(u_0)\leq \ell_1(u_0)-\delta_0,
$$
which is absurd. Hence there exists $t_0\geq 0$ such that
$$
\|(S_{t_0}u_0-v(p^+))_+\|_1< \delta_0\text{ and }\|(S_{t_0}u_0-v(p^-))_-\|_1< \delta_0.
$$
We have already proved that $\ell_1(S_{t_0} u_0)=0.$ We deduce that $\ell_1(u_0)=0$, and thus $\ell(u_0)=0.$

Consequently, assumption \textbf{(H')} entails that $\ell(u)=0$ for all $u\in\cA$.

\subsection{Stability of stationary periodic solutions in $L^1$}

We conclude this article by presenting some situations in which \textbf{(H)} or \textbf{(H')} hold true. We begin by explaining the linear case: assume that there exists a function $b\in\mathcal C_\text{per}(\R)$ such that 
$$
A(y,p)=b(y)p\quad\forall (y,p)\in[0,1]\times \R.
$$
In this case, the stability of periodic solutions is a consequence of a result of Adrien Blanchet, Jean Dolbeault, and Michal Kowalczyk (see \cite{BDK, BDK2}): indeed, set $\omega=-\mean{b}$, and let $\psi\in\mathcal C^2_\text{per}(\R)$ such that $\psi'=\mean{b}-b.$ Let $p\in\R$ be arbitrary, and let $u_0\in v(p)+ L^1_0$. Then, by linearity, $w(t)=:S_tu_0 -v(p)$ solves an equation on the type
$$\begin{aligned}
\p_t w + \p_y (b(y) w)-\p_{yy} w=0,\\
w_{|t=0}=w_0\in L^1_0(\R).
\end{aligned}
$$
It is then easily checked that the function $f$ defined by
$$
f(t,x)=w(t,x-\omega t)
$$
solves
\be
\p_t f(t,x) =\p_{xx} f (t,x)+ \p_x \left( \psi'(x-\omega t) f(t,x)\right).\label{eq:f}
\ee
This is precisely the case studied by Blanchet, Dolbeault and Kowalczyk. Let us recall briefly their method of analysis before stating their result. The first idea is to study the motion in the moving frame associated with the center of mass. Indeed, set
$$
\bar x (t):=\int_\R xf(t,x)\:dx.
$$
Then it can be easily proved, using the linearity of the evolution equation and the periodicity of $\psi'$, that
$$
\lim_{t\to\infty}\frac{d \bar x} {dt}(t)=\int_0^1 \psi' m,
$$
where $m$ is the unique probability measure on $[0,1]$ solving
\be\label{eq:m}
-m'' + \p_x((\omega+ \psi')m)= -m'' + \p_x(bm)=0.
\ee
Set $c:=\mean{\psi' m}$. The next idea is to perform a parabolic change of coordinates in the equation satisfied by $f$, in order to focus on the long-time behavior. Precisely, define $U$  such that
$$
f(t,x)=\frac 1{\sqrt{1+2t}} U\left( \log \sqrt{1+2t} , \frac{x-ct}{\sqrt{1+2t}}\right).
$$
Then $U$ solves an equation of Fokker-Planck type, with a penalization growing exponentially with time, and with coefficients which have fast oscillations for large times. Hence, this leads to the use of homogenization techniques, with the additional difficulty that the size of the oscillations in space depends on the time variable. An approximate solution is constructed thanks to a two-scale Ansatz. The convergence proof then relies on entropy dissipation methods. We are now ready to state their result.

\begin{prop}[Blanchet, Dolbeault, Kowalczyk] Let $w_0\in L^1_0\cap L^\infty(\R)$, and let $f$ be the solution of \eqref{eq:f} with initial data $f_{|t=0}=w_0$. Assume that there exists a constant $C_0>0$ such that
\be\label{momBDK}
\sup_{t\geq 0 }\frac{1}{(1+ 2t)^2}\int_\R |f(t,x)| \; (x-ct)^4\:dx\leq C_0.
\ee
Then there exists a constant $C_1$, depending only on $w_0$ and $C_0$, and a positive constant $\alpha$, depending only on $b$, such that for all $t\geq 0$,
$$
\|f(t) \|_1\leq \frac{C_1}{t^\alpha}.
$$
\label{prop:BDK}
\end{prop}
\begin{rem}
In fact, the result of  Blanchet, Dolbeault and Kowalczyk is a little more accurate than the above proposition. Indeed, they prove that  any non-negative solution $f$ of \eqref{eq:f} behaves asymptotically like 
$$
\frac{\int f_{|t=0}}{\sqrt{1+2t}}\;m(x-\omega t)\; h_\infty\left( \frac{x-c t}{\sqrt{1+2t}} \right),
$$
where $h_\infty$ is a Gaussian function.
In the present case, since $w_0\in L^1_0$, the solutions $f^+$ and $f^+$ of \eqref{eq:f} with initial data $f^\pm_{|t=0}= (w_0)_\pm$ have the same asymptotic behaviour, and thus $w=f^+-f^-$ decays towards zero.

\end{rem}

\vskip2mm

Consequently, in the linear case, assumption \textbf{(H)} is always satisfied, provided \eqref{momBDK} holds. However, if the flux $A$ is linear, standing shocks do not exist in general (see Remark \ref{rem:linear}). Thus, we now modify slightly the setting in order to use the results of the linear case, but in a non-linear context. Precisely, we now prove Proposition \ref{prop:linear}, which, as we have already stressed,  provides an explicit example of shock stability without any assumption of the initial data except \eqref{moments} (see Corollary \ref{cor:convex}).

Thus, let $A$ be a non-linear flux which satisfies the assumptions of Proposition \ref{prop:linear} for some $p\in\R$. Let $u_0\in v(p)+ L^1_0$ be arbitrary. Notice that we do not assume that 
$$
\| u_0 -v(p)\|_\infty\leq \eta,
$$
so that the use of the linear setting is not straightforward. The idea is to prove, using dispersion inequalities, that
\be
\lim_{t\to\infty}\|S_t u_0 - v(p) \|_\infty=0,\label{conv_infty}
\ee
provided $\| u_0 -v(p)\|_1$ is sufficiently small. If the above convergence is true,  there exists $t_0\geq 0$ such that for $t\geq t_0,$ for all $y$, $(S_t u_0)(y)\in[v(y,p)-\eta, v(y,p)+\eta] $, and thus
$$
A(y, S_t u_0(y))= A(y,v(y,p))+ b(y) w(t,y)\quad \forall t\geq t_0,\ \forall y\in\R,
$$
with $w(t)=S_t u_0 -v(p)$. Consequently, for $t\geq t_0$, $w$ solves a linear parabolic equation, and we can apply the previous analysis. The assumption on the moments of order four then becomes
\be
\exists t_0\geq 0,\quad \sup_{t\geq t_0}\frac{1}{(1+2t)^2}\int_{\R}|w(t,y)|\; (y-\gamma t)^4\:dy<\infty,\label{moments}
\ee
where $\gamma:=c-\omega$. Notice that $\gamma=\bar A'(p)$ in the present setting. Thus, we now focus on the proof of \eqref{conv_infty}.

As observed before, the function $w(t)=S_t u_0 - v(\cdot, p)$ ($t\geq 0$) is a solution of 
$$
\p_t w + \p_{y} B(y,w) -\p_{yy} w=0,
$$
where the flux $B$ is defined by
$$
B(y,\xi)=A(y,v(y,p)+ \xi) - A(y,v(y,p)),\quad y\in\R, \ \xi\in \R.
$$
The idea is to linearize the flux $B(\cdot, \xi)$ around $\xi=0$, and to use energy methods. Let
$$
\begin{aligned}
 b(y)=\p_v A(y,v(y,p)),\\
\tilde B(y,\xi)= A(y,v(y,p)+\xi) - A(y,v(y,p)) - b(y)\xi.
\end{aligned}
$$
Since $A\in W^{2, \infty}(\T\times\R)$, the flux $\tilde B$ is quadratic in a neighbourhood of $\xi=0$.
According to Proposition \ref{prop:bounds}, $w$ is bounded in $L^\infty([0,\infty)\times\R)$, and thus there exists a constant $C$ such that
$$
| \tilde B(y,w(t,y))|\leq C |w(t,y)|^2,\quad\forall t\geq 0,\ \forall y\in\R.
$$
The function $w$ solves
\be\label{eq:w_tildeB}
\p_t w + \p_{y} (b(y) w) -\p_{yy} w=- \p_y\tilde B(y,w).
\ee
Following an idea of Philippe Michel, St\'ephane Mischler and Beno\^it Perthame (see \cite{MMP}), we consider the invariant measure $m$, defined by \eqref{eq:m}.
Notice that $m=\p v /\p p$ in the present case, and there exists a positive constant $C\geq 1$ such that
$$
C^{-1}\leq m\leq C,\quad |\p_y m|\leq C\quad\text{a.e.}
$$
Moreover, according to \cite{MMP}, the following identity holds 
$$
\p_t \left( m \left| \frac{w}{m} \right|^2 \right) + \p_y \left( m \left| \frac{w}{m} \right|^2 \right)- \p_{yy} \left( m \left| \frac{w}{m} \right|^2 \right)=-2 m \left| \p_y \left( \frac{w}{m} \right) \right|^2 - 2\frac{w}{m}\p_y\tilde B(y,w).
$$
Integrating the above equation on $\R$, we obtain
\begin{eqnarray*}
 \frac{1}{2}\frac{d}{dt}\int_{\R}m\left| \frac{w}{m} \right|^2 + \int_{\R}m\left|\partial_y\frac{w}{m} \right|^2 &\leq& C \int_{\R}|w|^2\left|\partial_y \frac{w}{m} \right|\\&\leq& C \left\|\frac{w}{m}  \right\|_{L^4(m(y)dy)}^2\left\|\nabla\frac{w}{m}  \right\|_{L^2(m(y)dy)}.
\end{eqnarray*}
Notice that for all $p\in[1,\infty)$, the $L^p$ norm is equivalent to the $L^p(m(y)dy)$ norm.
We now use the following Poincar\'e inequality: there exists a positive constant $C$, such that for all $\phi\in L^1(\R)\cap H^1(\R)$, there holds
\be
\|\phi\|_{L^4(\R)}\leq C \| \phi'\|_{L^2(\R)}^{1/2}\| \phi\|_{L^1(\R)}^{1/2}.
\ee
Taking $\phi=w/m$,  we are led to
\begin{eqnarray*}
 \frac{1}{2}\frac{d}{dt}\int_{\R}m \left| \frac{w}{m} \right|^2 + \int_{\R}m \left| \nabla \frac{w}{m} \right|^2 &\leq& C \| w\|_{L^1(\R)}\int_{\R}m \left| \nabla \frac{w}{m} \right|^2\\
& \leq&  C \| w_0\|_{L^1(\R)}\int_{\R}m \left| \nabla \frac{w}{m} \right|^2 .
\end{eqnarray*}

Now, if $\| w_0\|_{L^1(\R)}$ is sufficiently small, we obtain
\be
\frac{d}{dt}\int_{\R}m \left| \frac{w}{m} \right|^2 + \int_{\R}m \left| \nabla \frac{w}{m} \right|^2 \leq 0.\label{in:decayL2}
\ee
We then proceed as in \cite{SerreHandbook} (Paragraph 1.1): using the Nash inequality together with the decay of the $L^1$ norm, we deduce that for all $t\geq 0$
$$
\left\| \frac{w(t)}{m}\right\|_{L^2(m)}\leq C \|w(t) \|_{L^1}^{2/3} \left\|\nabla\frac{w(t)}{m}  \right\|_{L^2(m)}^{1/3}\leq C \|w_0\|_{L^1}^{2/3} \left\|\nabla\frac{w(t)}{m}  \right\|_{L^2(m)}^{1/3},
$$
and thus we infer
$$
\frac{d}{dt}\int_{\R}m \left| \frac{w}{m} \right|^2 + \frac{C}{\|w_0\|_{L^1}^4}\left(\int_{\R}m \left| \frac{w}{m} \right|^2 \right)^3\leq 0.
$$
Integrating the above differential inequality, we obtain eventually
$$
\|w(t)\|_{L^2(\R)}\leq C \left\|\frac{w}{m} \right\|_{L^2(m)}\leq C\frac{\| w_0\|_{L^1}}{t^{1/4}}.
$$
Thus the $L^2$ norm decays with an algebraic rate. 

We now use a parabolic regularity result, from which the decay of the $L^\infty$ norm immediately follows. The key point lies in the following inequality: there exists a constant $C$ such that for all $t\geq 1,$
$$
\|w(t) \|_\infty\leq C \| w(t-1)\|_2.
$$
Indeed, $w$ satisfies an equation of the type
\be
\p_t w + \p_y(a(t,y) w(t,y)) -\p_{yy} w=0,\label{eq:w_linear}
\ee
where the coefficient $a$ is bounded in $L^\infty([0,\infty]\times\R)$. Using an energy estimate, it can be easily proved that there exists a positive constant $\alpha$, depending only on $\| a\|_\infty$, such that if $W$ is any solution of \eqref{eq:w_linear}, then for any $t\geq s\geq 0$
$$
\| W(t)\|_2\leq e^{\alpha (t-s)} \| W (s)\|_2.
$$
Moreover, according to Harnack's inequality, there exists a constant $C$ such that for any non-negative solution $W$ of \eqref{eq:w_linear}, for all $t\geq 0$,
$$
W(t,y)\leq C\inf_{z\in [y-1,y+1]} W(t+1,z)\:dz.
$$
And if $t\geq 1$,
\begin{eqnarray*}
 \inf_{z\in [y-1,y+1]} W(t+1,z)\:dz&\leq &\left(\frac{1}{2}\int_{y-1}^{y+1} W^2(t+1,z)\:dz\right)^{1/2}\\&\leq &C \left(\int_{y-1}^{y+1} W^2(t-1,z)\:dz\right)^{1/2}\\&\leq &C\| W(t-1)\|_{L^2(\R)}.
\end{eqnarray*}
Now, let $t\geq 1$ be arbitrary, and let $W_1^0:=(w(t-1))_+$, $W_2^0:=(w(t-1))_-$. For $s\geq t-1$, consider the solution $W_i$ of \eqref{eq:w_linear} such that $W_{i|s=t-1}=W^0_i$. The functions $W_i$ are non-negative by the maximum principle. Consequently, for all $y\in \R$, we have
$$
W_i(t,y)\leq C \|W_i^0\|_2.
$$
Since $w=W_1-W_2$, we deduce that
$$
\sup_{y\in\R}|w(t,y)|\leq C \|w(t-1)\|_2.
$$
The decay of the $L^\infty$ norm follows. This concludes the proof of Proposition \ref{prop:linear}. \qed

\vskip3mm

In fact, the decay of all $L^p$ norms for $p\in(1,\infty]$ is a general property, which is true even when the flux $A$ does not satisfy the assumptions of Proposition \ref{prop:linear}. However, if the flux $A$ is not linear in a neighbourhood of $v(p^\pm)$, then we are unable to conclude to the stability of periodic solutions. Let us now explain briefly where the difficulty lies: a natural idea would be to treat the term $\tilde B$ as a perturbation in \eqref{eq:w_tildeB}, and to write a Duhamel formula of the type
$$
w(t)=S^0_t w_0 + \int_0^t S^0_{t-s}\left[\p_y \tilde B(\cdot, w(s))\right]\:ds,
$$
where $S^0$ is the (linear) semi-group associated with the equation$$\p_t w + \p_y(bw) - \p_{yy} w=0.$$
This is exactly the method used in \cite{SerreHandbook} in order to prove the stability of constants in the viscous model. Thanks to the results of \cite{BDK}, it is already known that $S^0_t w_0$ decays in $L^1$ as $t\to\infty$. 
However, the method used in \cite{SerreHandbook} cannot be used here, essentially because of the complicated dependance of the constant $C_1$ appearing in Proposition \ref{prop:BDK} on the function $w_0$. Indeed, looking carefully at the proof in \cite{BDK}, it can be checked that
$$
C_1\leq C \|w_0\|_1\left( \int_\R \frac{w_0^+}{\| w_0^+\|_1}\ln \left( \frac{w_0^+}{\| w_0^+\|_1 h_\infty} \right) +  \int_\R \frac{w_0^-}{\| w_0^-\|_1}\ln \left( \frac{w_0^-}{\| w_0^-\|_1 h_\infty} \right) + C_0\right)^{1/2},
$$
where $h_\infty$ is a normalized gaussian function, $w_0^+, w_0^-$ are the positive and negative real parts of $w_0$, and
$$
C_0:=\sup_{t\geq 0}\frac{1}{(1+2t)^2}\int_\R |S_t w_0|(x-\bar A'(p)t)^4.
$$
The dependance of $C_1$ with respect to $C_0$ and the relative entropies of $w_0^+$ and $w_0^-$ is disastrous for the use of the Duhamel formula: indeed, one has to control, for instance,
$$
\int_\R \frac{(\p_y \tilde B(\cdot, w(s)))_+ }{\| \p_y \tilde B(\cdot, w(s)))_+\|_1}\ln \left( \frac{\p_y \tilde B(\cdot, w(s)))_+}{\|\p_y \tilde B(\cdot, w(s)))_+\|_1 h_\infty} \right).
$$
\vskip1mm

Consequently, another approach must certainly be chosen in order to prove the stability of periodic solutions in the general case. Given the complexity of the proof in the mere linear setting (see \cite{BDK}), this question  goes beyong the scope of this article. Also, we emphasize that it is not clear that the methods of \cite{BDK} can be adapted to a nonlinear setting: indeed, the proof of convergence relies on the use of entropy dissipation techniques, which are more adapted to the linear case. We refer to \cite{EVZ} for additional results and techniques concerning the asymptotic behaviour of non linear viscous conservation laws in the homogeneous case.

\section*{Appendix A - Proof of Lemma \ref{lem:convex}}

Assume that the flux $A$ is convex, and let $p_1,p_2\in \R$ such that $p_1\neq p_2$, and let $ \lambda\in(0,1).$ In the following, we set
$$\begin{aligned}
   	v_i(y)=v(y,p_i),\quad i=1,2,\\
	w=\lambda v_1 + (1-\lambda ) v_2,\quad p=\lambda p_1 + (1-\lambda) p_2,\\
	u(y)=v(y,\lambda p_1 + (1-\lambda) p_2).
  \end{aligned}
$$
By definition of $v(\cdot, p)$ and of the homogenized flux $\bar A$, we have
$$
\begin{aligned}
 -v_i' + A(y,v_i(y))=\bar A(p_i),\\
-u' + A(y,u(y))=\bar A(\lambda p_1 + (1-\lambda )p_2).
\end{aligned}
$$
Consequently, using the convexity of the flux $A$, we deduce that for all $y\in \T^N$,
\begin{eqnarray}\label{in:convex1}
- w'(y) + A(y, w(y))&\leq& -w'(y) + \lambda  A(y,v_1(y)) + (1-\lambda) A(y,v_2(y))\\\nonumber&=&\lambda \bar A(p_1) + (1-\lambda) \bar A(p_2).
\end{eqnarray}

Assume that $\bar A(\lambda p_1 + (1-\lambda) p_2)> \lambda \bar A(p_1) + (1-\lambda) \bar A(p_2)$, and write $u,w$ as $$
u= p+ f',\quad w= p + g',
$$
with $f,g\in\mathcal C^2_{\text{per}}(\T^N).$ Since $f$ and $g$ are defined up to the addition of constants, we can assume that $f<g$ almost everywhere. Moreover, notice that
$$
\sup_{y\in \T^N}\left( -g''(y) + A(y,p+g'(y)) \right)<\inf_{y\in \T^N}\left( -f''(y) + A(y,p+f'(y)) \right).
$$
Thus there exists $\alpha>0$ such that
$$
-g''+ A(y,p+g'(y)) + \alpha g \leq -f''+ A(y,p+f'(y)) + \alpha f.
$$
Hence, by the maximum principle, we infer that $g\leq f$, which is absurd. Thus
$$
\bar A(\lambda p_1 + (1-\lambda) p_2)\leq \lambda \bar A(p_1) + (1-\lambda) \bar A(p_2).
$$

\vskip2mm

If the flux $A$ is strictly convex, then inequality \eqref{in:convex1} is strict for all $y\in \T^N$ (remember that the family $v(y,p)$ is strictly increasing with $p$ for all $y\in \T^N$). Consequently, the same argument as above leads to
$$
\bar A(\lambda p_1 + (1-\lambda) p_2)<\lambda \bar A(p_1) + (1-\lambda) \bar A(p_2)
$$

\section*{Appendix B}

\begin{lemA}
Let $w\in L^1\cap L^\infty(\R) $ such that $w'\in L^2(\R)$ and $w''\in L^1_\text{loc}(\R)$. Assume that $w$ is such that
$$
\lim_{R\to\infty}\int_\R\sgn(w(y)) w''(y) \phi\left( \frac{y}{R} \right)\:dy=0
$$
for all $\phi\in\mathcal C^\infty_0(\R)$ such that $\phi\equiv 1$ in a neighbourhood of zero. 
Then
$$
\lim_{\delta\to 0} \frac{1}{\delta}\int_{\R} |w'|^2\mathbf 1_{|w|< \delta}=0.
$$
As a consequence, 
$$
\p_{yy} |w| = \sgn(w)  w''\quad\text{in }\mathcal D'(\R).
$$

\label{lem:tech}
\end{lemA}

\begin{proof}
For $\delta>0$, let
$$
\psi_\delta(x):=\left\{ 
\begin{array}{ll}
\sgn(x)&\text{ if } |x|\geq \delta,\\
\ds\frac{x}{\delta}&\text{ else.}
\end{array}
\right.
$$ 
Then $$\psi'_\delta(x) = \frac{1}{\delta}\mathbf 1_{|x|< \delta},$$
and for all $R>0$, we have, using the chain rule
$$
\int |w'|^2 \psi'_\delta(w) \phi_R= - \int w'' \psi_\delta(w) \phi_R - \int w' \psi_\delta(w) \phi_R ',
$$
where $\phi_R=\phi(\cdot/R).$

Since $w'\in L^2,$ we infer
$$
\left| \int w' \psi_\delta(w) \phi_R'\right|\leq \int |w'|\; \left| \phi_R' \right|\leq R^{-1/2} \|w' \|_{L^2} \| \phi'\|_{L^2}.
$$
Thus the above term vanishes as $R\to\infty$, uniformly in $\delta$.

On the other hand,
$$
\lim_{\delta\to 0} \int w'' \psi_\delta(w) \phi_R= \int w'' \sgn(w) \phi_R,
$$
and the right-hand side vanishes as $R\to\infty$ by assumption. We deduce that
$$
\lim_{R\to\infty} \limsup_{\delta\to 0 }\int |w'|^2 \psi'_\delta(w) \phi_R=0.
$$

Now, since the integral $\int |w'|^2 \psi'_\delta(w) \phi_R$ is non-negative and increasing with respect to $R$, we deduce that
$$
\lim_{\delta\to 0 }\int |w'|^2 \psi'_\delta(w) \phi_R=0\quad \forall R,
$$
and thus the first part of the lemma is proved.
\vskip1mm
Consider $S_\delta \in  W^{2,1}_\text{loc}(\R)$ such that
$$
S_\delta'=\psi_\delta\text{ and } S_\delta(0)=0,
$$
where the function $\psi_\delta$ was defined earlier. Then
$$
S_\delta(w)\to |w |\quad\text{in }L^1_\text{loc}(\R),
$$
and according to the chain rule,
$$
\p_{yy} S_\delta(w)=w'' \psi_\delta(w) + |w'|^2\frac{\mathbf 1_{|w|\leq \delta}}{\delta}.
$$
Passing to the limit in the sense of distributions in the above equality yields
$$
\p_{yy}|w| =w'' \sgn(w).
$$

\end{proof}

\section*{Acknowledgements}
I am very grateful to Denis Serre, for encouraging me to work on these questions in the first place, and for very interesting and fruitful discussions. I also wish to thank Jean Dolbeault, for his helpful insight of  the long-time behaviour of linear diffusion equations.

\bibliography{longtime_VSCL}

\end{document}